\documentclass[11pt, a4paper]{article}
\usepackage{adjustbox}
\usepackage{srcltx}
\usepackage{amsmath,latexsym,amsfonts,amssymb}
\usepackage{pstricks}
\usepackage[latin1]{inputenc}
\usepackage{epsfig}
\usepackage{fancyhdr}
\usepackage{graphics,graphicx}
\usepackage{psfrag}
\usepackage{subfigure}
\usepackage{xspace}
\usepackage{enumerate}
\usepackage{rotating}
\usepackage{multirow}
\usepackage[linesnumbered,ruled,vlined]{algorithm2e}
\usepackage{pifont}
\usepackage{tikz}
\usepackage{pgfplots}
\usepackage[round]{natbib}
\usepackage{color}
\usepackage{tablefootnote}
\usepackage{authblk}
\usepackage[breaklinks=true]{hyperref}
\usepackage{breakcites}

\newtheorem{propo}{\bf Proposition}[section]

\newtheorem{rem}{\bf Remark}[section]

\allowdisplaybreaks[1]

\def\BP{B\&P\xspace}
\def\GCPP{GCCP\xspace}
\def\GCSC{MGCSC\xspace}

\usepackage[normalem]{ulem}

\newcommand{\msout}[1]{\text{\sout{\ensuremath{#1}}}}
\newcommand{\rv}[1]{\relax\ifmmode\azul{#1}\else\azul{\textbf{#1}}\fi}
\newcommand{\rvv}[2]{\rojo{\ifmmode\msout{#1}\else\sout{#1}\fi} \azul{\ifmmode#2\else\textbf{#2}\fi}}

\definecolor{armygreen}{rgb}{0.19, 0.53, 0.43}
\definecolor{armygreen2}{rgb}{0.99, 0.53, 0.93}
\definecolor{armygreen3}{rgb}{0.93, 0.53, 0.19}

\newcommand{\JPA}[2]{{\color{armygreen}#1\ifmmode\msout{#2}\else\sout{#2}\fi}}
\newcommand{\AMR}[2]{{\color{armygreen2}#1 \ifmmode\msout{#2}\else\sout{#2}\fi}}

\begin{document}

\title{ \vspace*{-.3cm}
A Branch-and-price procedure for clustering data that are graph connected}

\author[(1)]{Stefano Benati}
\author[(2)]{Diego Ponce}
\author[(2)]{Justo Puerto}
\author[(3)]{Antonio M. Rodr\'iguez-Ch\'ia}

\affil[(1)]{\small{Dipartimento di Sociologia e Ricerca Sociale, Universit\`{a} di Trento, Italy, stefano.benati@unitn.it}}

\affil[(2)]{\small{IMUS and Departamento de Estad\'istica e Investigaci\'on Operativa, Universidad de Sevilla, Spain, \{dponce,puerto\}@us.es}}


\affil[(3)]{\small{Departamento de Estad\'istica e Investigaci\'on Operativa, Universidad de C\'adiz, Spain, antonio.rodriguezchia@uca.es}}


\date{}

\maketitle
\thispagestyle{empty}

\vspace*{-1.5cm}
\begin{abstract}
This paper studies the Graph-Connected Clique-Partitioning Problem (\GCPP), a clustering optimization model in which  units are characterized by both individual and relational data. This problem, introduced by \cite{BPR17} under the name of Connected Partitioning Problem,  shows
that the combination of the two data types improves the clustering
quality in comparison with other methodologies. Nevertheless, the resulting optimization problem is difficult to solve; only small-sized instances can be solved exactly, large-sized instances require the application of heuristic algorithms.
In this paper we improve the exact and the heuristic algorithms previously proposed. Here, we provide a new Integer Linear Programming (ILP) formulation, that  solves larger instances, but at the cost of using an exponential number of variables. In order to limit the number of variables necessary to calculate the optimum, the new ILP formulation is solved implementing a branch-and-price (\BP) algorithm. The resulting pricing problem is itself a new combinatorial model: the Maximum-weighted Graph-Connected Single-Clique problem (\GCSC), that we solve testing various Mixed Integer Linear Programming (MILP) formulations and proposing a new fast ``random shrink'' heuristic.
In this way, we are able to improve the previous algorithms: The \BP
method outperforms the computational times of the previous MILP algorithms and the new random shrink heuristic, when applied to \GCPP, is both faster and more accurate than the previous heuristic methods. Moreover, the combination of column generation and random shrink is itself a new MILP-relaxed matheuristic that can be applied to large instances too. Its main advantage is that all heuristic local optima are combined together in a restricted MILP, consisting in the application of the exact B\&P method but solving heuristically the pricing problem.\end{abstract}

\textheight 23cm%

\section{Introduction}\label{sec:1}
We consider a clustering problem in which units are characterized by both individual and relational data. Individual data take the form of a matrix F of $n$ rows, representing units, and $m$ columns, representing features that are measured for individuals. Individual data are then complemented by relational data, for example representing friendship, communication, co-participation and so on. Relational data are described as an undirected graph $G = (V, E)$ in which $V$ are the units, $| V | = n$, and there is an edge $e_{ij} \in E$ if and only if there is a relation between $i,j \in V$.  The data structure that combines the graph $G = (V, E)$ with the data matrix $F$ forms the triplet $G = (V, E, F)$, called {\it attributed graph}, see \cite{BCMM-2015}.

The simplest method of clustering attributed graphs is projecting the relational data into the individual data, or vice versa, the individual into the relational. In the former case, a dissimilarity measure between units $i$ and $j$ is calculated using {\it both} individual measures of F {\it and} the existence/non-existence of an arc $i,j$, \cite{CL-2012, Z-2012}. In the latter case, the matrix F is used to calculate a distance $d_{ij}$ attached to an existing arc $e_{ij}$, and to convert the {\it unweighted} graph into a {\it weighted} one, \cite{NAD-2003}. In both cases, the problem is reduced to standard clustering or graph partitioning problems respectively, and any solution methods for those problems can be applied. The interested reader is referred to \cite{GAMBELLA2021807} for a recent survey.
Alternatively, the two data structures are kept separate and then one can formulate an optimization model to determine the best classification. The optimization model must be formulated in such a way that it takes into account that relational data give additional information about the similarity between units. That is, the objective function or the constraints set must reflect some connectivity requirement. In \cite{BPR17}, clustering with graph-connected units is modeled as a combinatorial problem in which the most similar groups are evaluated through the clique partition of individual data, namely, induced by the information in F but with the additional constraint that those cliques  must be additionally connected through the underlying graph $G$, representing the relational data. The problem so formulated has been called the Graph-Connected Clique-Partitioning Problem (\GCPP). \cite{BPR17} shows that this model is superior than classical clustering methods in finding true clusters.

More formally, the \GCPP consists in the following. An attributed graph $G = (V, E, F)$ is given, so that the similarity/dissimilarity distances $c_{ij}$ between all pairs $i,j$ can be calculated using only the information contained in F, see \cite{BPR17} for further details on its computation. These $c_{ij}$
are used to formulate the objective function of a clique partitioning problem as done in \cite{GW-1989}. Relational data $E$ are used imposing that the optimal clique partition $\Pi = \{V_1, \ldots, V_p\},\, 1\le p\le n,$ must be composed of components $V_{k} \subseteq V, \, k=1,\dots, p,$ {\it connected} trough the arcs of $E$. Some empirical experiments have shown that combining the two data sources through this model improves the clustering quality.

In \cite{BPR17}, exact and heuristic methods are proposed to solve the \GCPP. Exact solutions are calculated though different MILP models. Those models differ on how they impose connectivity through a set of linear constraints. Connectivity can be imposed through flow conservation laws, or using constraints describing the forest/tree decomposition, and models can be strengthened with valid inequalities. Some formulations are more advantageous than others, but, in all cases, only problems of moderate size can be solved exactly. Various implementations of local search heuristics are tested as well, but, even though those methods find reasonably accurate solutions, their computational times are high. Therefore, it is worth exploring the possibility of  improving on these previous findings.

In this paper we are proposing three new methods, one exact procedure and two heuristics, to solve the \GCPP. The exact method is based on a branch-and-price algorithm (\BP), a technique that has been proved successful when applied to other clustering problems, \cite{Mehrotra1998, Aloise2010}. The interested reader can also see \cite{LubbeckeMarcoE2005STiC}, \cite{GualandiM13}, and the references therein to gain further insight into column generation techniques.

The first step of the algorithm is to formulate \GCPP as a Set Partitioning (SP) problem. In the SP, a binary variable $y_S$ is defined for every feasible subsets $S\subseteq V$, and then a set of linear constraints defines the feasible solutions. Obviously, the straight solution of the model is impeded by the exponential number of variables, $O(2^n)$, but actually there is no need to consider them all just from the beginning. Rather, one can start with a MILP formulation including only a few of the $y_S$'s, solve the problem, and then adding new variables only after the result of the reduced cost test. The reduced cost test relies on the exact or heuristic solution of a new combinatorial problem, the Maximum-weighted Graph-Connected Single-Clique problem (\GCSC). We formulate the \GCSC as a MILP model, testing the effectiveness of various formulations. Moreover, as it is important to find a solution quickly, thus a fast, greedy-like constructive heuristic has been developed, inspired by the noising method proposed in \cite{Charon2006}. As a result, it has been found that this heuristic can be applied to the \GCPP as well, providing a faster and more accurate algorithm than local search heuristics.
In addition, a MILP-relaxed matheuristic procedure is developed that combines
the quickness of previously described heuristic  with the accuracy of the column generation developed for the exact method. We refer the reader to \cite{Raidl2015} and the references therein for alternative successful combinations of column generation and heuristics. Finally, we found that our implementation of \BP, the heuristic and matheuristic approaches developed in this paper are respectively improvements of the previous exact and heuristics solution procedures as they calculate faster their respectively optimal or approximate solutions.

The paper is structured in \ref{sec:7} sections, the first being this introduction. In Section \ref{sec:2}, we provide a formal definition of the problem and its formulation as a SP with an exponential number of variables. In Section \ref{sec:pricing}, we discuss the pricing problem consisting of a new combinatorial problem, the \GCSC, so we discuss how to calculate its optimal solution. In Section \ref{sec:4}, we describe a fast heuristic for an approximate solution of both \GCPP and \GCSC, based on greedy, but enhanced through the use of some random steps. Also a MILP-relaxed matheuristic, capable of handling very large instances with good accuracy, is proposed. Section \ref{sec:5}  is devoted to describing some details of the \BP which are not included in Section \ref{sec:2}  for the ease of compactness. In Section \ref{s:comp-test}, we report our computational analysis, comparing the exact methods to solve the \GCPP by \BP through different formulations of the pricing
problem and testing the performance of the heuristics too. The paper ends with some remarks on future research directions.


\section{Problem definition and set partitioning formulation\label{sec:2}}

In this section, we formally define the \GCPP.
Let $V = \{1, \ldots, n\}$ be a set of units and $C=(c_{ij})_{i,j\in V}$ a measure of similarity/dissimilarity between units, with $c_{ij} < 0$ denoting similarity, dissimilarity otherwise.
Assume that units of $V$ are embedded in a graph $G=(V,E)$, whose edges $e_{ij} \in E$ describe links between $i,j \in V$. Given $Q \subseteq V$, let  $G[Q] = (Q, E[Q])$ be the subgraph induced by $Q$, i.e., the graph with edges  $e_{ij} \in E[Q]$ iff  $i,j \in Q$ and $e_{ij} \in E$. We say that $Q \subseteq V$ is connected if $G[Q] = (Q, E[Q])$, i.e., the subgraph induced by $Q$, is a connected subgraph.

The  goal of \GCPP is to find a partition $\Pi = \{V_1, \ldots, V_p\}$ of $V$ (with parameter $p$ not fixed in advance, i.e., $1\le p \le n$), such that any $V_{k},  k=1,\ldots,p,$ is connected and minimizing the objective function:
\begin{equation*}
f(\Pi) = \sum_{k = 1}^p \sum_{i,j \in V_k} c_{ij}.
\end{equation*}

Hence, \GCPP can be formulated as follows:
\begin{eqnarray*}
& \displaystyle \min_{\Pi \in {\cal P}}  & \; f(\Pi)  \\
                                   & \mbox{ s.t. } & V_{k} \mbox{ is connected for all $V_{k} \in \Pi$,}
\end{eqnarray*}
where ${\cal P}$ is the set of all the partitions of $V$.

As \GCPP is in minimization form, units $i$ and $j$ for which $c_{ij}$ is negative will tend to be in the same group, while units for which $c_{ij}$ is positive will tend to be in different groups. Introducing a connection constraint between units implies that even though a unit can be similar to several others, it can be clustered only to the connected units.

In \cite{BPR17}, \GCPP has been formulated and solved with exact and heuristic methods. Exact methods are some MILP formulations based on the Clique Partition problem with connection constraints. Heuristic methods are the improved local search heuristics Variable Neighborhood Search (VNS) and Random Restart (RR). In this paper, we introduce a new MILP formulation with an exponential number of variables that will be solved through column generation, embedded in a branch-and-price algorithm. Next we introduce two new heuristic procedures: A constructive heuristic based on random shrink and a MILP-relaxed matheuristic based on approximated column generation.
All new methods are improvements over the old ones, as the exact method improves the computational times and the maximum size of the solved instances, while the heuristics improve the optimum approximation for a given computational time.

\subsection{The Set Partitioning formulation\label{ss:setpartitioning}}

In this section, a new formulation of \GCPP is introduced, in which an exponential number of variables are needed.  Suppose that we can list all connected subsets $S$ of $V$: Let ${\cal S} = \{S \: | \: S \subseteq V, G[S] \mbox{ is connected}\}$  and let $c_S=\sum_{i\in S} \sum_{j\in S: j>i} c_{ij}$. Let $y_S$ be a binary variable defined for all $S \in {\cal S}$ such that:
$$y_S= \begin{cases}
1, & \mbox{if } S \in {\Pi},  \\
0, & \mbox{otherwise.}
\end{cases}$$
Hence, \GCPP can be formulated as follows:
\begin{eqnarray*}
\mathbf{(MP)}  & \min  & \displaystyle\sum_{ S \in {\cal S}} c_S y_S \\
&s.t.& \displaystyle \sum_{S \in {\cal S}\, : \, i \in S} y_S  =  1, \quad \forall i \in V,\\
& &y_S\in\{0,1\}, \quad \forall\,S \in \cal{S}.
\end{eqnarray*}
The problem constraints ensure that a unit is included in exactly one cluster, so that subsets $S$ must form a partition $\Pi$. The value of a partition is given by the problem objective function. The drawback of (MP) is that it contains an exponential number of binary variables to explicitly define $\mathcal{S}$. Hence, we  consider its restricted  version. The idea is to formulate (MP) with only a fraction of the $y_S$ variables. Then, solving its linear relaxation, we can obtain reduced costs for the absent variables $y_S$ and determine whether a new variable/column $y_S$ is to be introduced in the relaxed and restricted (MP), or the current solution is optimal for that problem. Branching is applied each time a  not integral solution is found until optimality is proved. The reader is referred to the following works and the references therein for further details on the following topics: \cite{Desrosiers2005}, for a precise presentation about column generation; \cite{Barnhart96branch-and-price:column} for a detailed explanation about branch-and-price; and to \cite{Deleplanque2020}, for a recent application of those techniques. A pseudocode of this method is provided in Algorithm \ref{algo:CG} and explained in detail in Section \ref{sec:5}.

\begin{algorithm}[htb]
\DontPrintSemicolon
\KwIn{An instance of \GCPP with data $C, G = (V,E)$.}
\KwOut{An optimal partition $\Pi$ of $V$. }
$\mathbb{S} \gets \mbox{Initiate($C,G$)}$\;
$\mbox{optimality} \gets \mbox{false}$\;
$\mbox{node} \gets \mbox{root node}$\;

\While{optimality = false}{
$( \gamma^*, y^*) \gets  \mbox{Solve}\mathbf{((Relaxed MP )}_{\mathbb{S}}$, node)\;
$S \gets \mbox{Solve\_Pricing\_Problem}( \gamma^*$, node)\;
\If{$\bar c(y_S) < 0$}{
		$\mathbb{S} \gets \mathbb{S} \cup S$\;
		}
\Else{
		\If{$y^* \mbox{ integral }$}{
		        \If{$\mbox{Upper\_Bound}(y^*)$}{
			          $\mbox{optimality} \gets \mbox{true}$\;
			 }
			 \Else{
			     $\mbox{node} \gets \mbox{Next\_Node(MP)}$
			 }
		 }
		\Else{
			\If{$\mbox{Lower\_Bound}(y^*)$}{
			$\mbox{optimality} \gets \mbox{true}$\;
			}
			\Else{
			      $\mbox{Branch}(y^*)$\;
			      $\mbox{node} \gets \mbox{Next\_Node(MP)}$
			}
		}			
			
	}

} 
\caption{{\sc \BP for \GCPP}}
\label{algo:CG}
\end{algorithm}

\subsection{Relaxed restricted master problem}

Here we explain the solution procedure of the relaxed master problem at the root node. The same procedure is applied in the remaining nodes. The particularities involved in the solution of branched nodes can be found in Section \ref{sec:5}.

Let $\mathbb{S} \subseteq \cal{S}$ be a subset of all the feasible clusters. The relaxed and restricted master problem is:
\begin{eqnarray*}
\mathbf{(Relaxed MP)}_{\mathbb{S}} \hspace*{1cm} & \min&\displaystyle\sum_{S\in \mathbb{S}} c_Sy_S\hspace*{4cm} \textbf{Dual Multipliers} \\
&s.t.& 
\displaystyle \sum_{S \in \mathbb{S}\, : \, i \in S}y_S = 1, \quad  \forall i\in V, \hspace*{1.85cm} \gamma_i  \text{ unrestricted}\\
&  & y_S \ge 0, \quad \forall\,S \in \mathbb{S}.
\end{eqnarray*}
%
Observe that the dual multipliers associated with each constraint are emphasized in the right-hand side of the formulation above.

The dual of the relaxed and restricted master problem is
\begin{eqnarray*}
\mathbf{(DP)}_{\mathbb{S}}  \hspace*{1cm} &\max&\displaystyle 
 \sum_{i=1}^n \gamma_i\\
&s.t.& \sum_{i\in S} \gamma_i \le c_S, \quad \forall S\in \mathbb{S},\\
&   &\gamma_i \mbox{ unrestricted,} \quad \forall i\in V.
\end{eqnarray*}

Given an optimal solution $\gamma^*$ of $\mathbf{(DP)}_{\mathbb{S}}$,
we can obtain the reduced cost of an absent variable $y_S$ of the master problem as:
$$ \bar c(y_S)= c_S- \sum_{i\in S} \gamma_i^*.$$
If it can be proved that the reduced costs of all the missing variables are nonnegative, then the master problem is solved to optimality. Otherwise, any variable $y_S$ with $\bar c(y_S)<0$ induces a new column to be included in $\mathbf{(Relaxed MP)}_{\mathbb{S}}$ to (possibly) improve the incumbent solution. We refer to the pricing problem as the problem of finding a cluster $S \in \cal{S}$ such that
 $c_S - \sum_{i\in S}\gamma^*_i<0$, or to prove that it does not exist. If, after solving the pricing problem, one or more new variables $y_S$ are introduced in $\mathbf{(Relaxed MP)}_{\mathbb{S}}$,  then it is solved again. Otherwise, the relaxed master problem is solved to optimality.

\section{The pricing problem\label{sec:pricing}}

Step 6 of Algorithm \ref{algo:CG}, i.e., the solution of the pricing problem, is an important step. The problem consists in answering the question
$$\mbox{``Is $c_S-\sum_{i\in S}\gamma^*_i<0$ for some $S\in \mathcal{S}$?''}$$
To respond to the query we define a new combinatorial problem on the graph $G=(V,E)$ where inputs are  the  costs  $c_{ij}$ associated to each pair of nodes $i,j \in V$ and node weights $-\gamma^*_i$ for all $i \in V$. Then, the Maximum-weighted Graph-Connected Single-Clique (\GCSC) on $G$ consists in: Given a graph $G=(V,E)$ with weights associated with  each pair of nodes and each individual node, find a connected subset of $V$, minimizing the sum of both node weights and pairs-of-nodes weights. The reader should observe that in our application we solve minimization problems since weights can be positive and negative. This problem is related with the prize collecting Steiner tree problem, \cite{Ljubic06}, and the maximum weight connected subgraph problem, \cite{Alvarez-Miranda2013}, although in both cases the graph structure, weights and the objective function are different. The \GCSC  reduces to the maximum-weighted  clique problem when $G$ is a complete graph, therefore the former is trivially $\mathcal{NP}$-hard, \cite{Balas1987}.

The problem described above can be formulated in each pricing iteration at the root node of the master problem. As before, we leave the necessary branching modifications to Section \ref{sec:5}.

In spite of its exponential worst-case complexity, Step 6 can be implemented in such a way to maintain an efficient computation. In fact, it is not necessary to find the optimal $S$, that is, to calculate the exact minimum $\bar c(y_S)$. It is sufficient to find any $S\in \mathcal{S}$ for which $\bar c(y_S) < 0$ (and even more than one of such $S$ if possible). Therefore, we can solve \GCSC using a heuristic method and only when the heuristic fails, we calculate its exact solution. At the end of the algorithm, the exact solution of the pricing problem is surely needed to certify optimality in (MP), that is, proving that all missing $S$'s are such that $\bar c(y_S) \ge 0$. Nevertheless, before that, hopefully a large amount of required variables are detected heuristically.

In the next subsections, we propose some MILP models to solve \GCSC. The main differences among models are the type of constraints that impose connectivity.

\subsection{Flow-based formulation\label{ss:flowsin}}

The idea behind this formulation is that if a set $S \subseteq V$ is connected, then
a source node can send a unit of flow to  any node of $S$ using the auxiliary network induced by $S$. 
Let $G_D = (V ,A)$ be a digraph with set of arcs,  $A$, so defined: Two arcs  $(i,j)$ and $(j,i)$ for every edge $e_{ij} (=e_{ji})\in E$. For each subset $S\subset V$ one of  its nodes is assumed to be a source and all the remaining nodes ask for a unit of flow that must be sent from that source. 
Then, an objective function is minimized with respect to a node set $S$, but constraints will try to establish a flow from the source to the nodes of $S$. If a flow is permissible, then those nodes are connected and $S$ is feasible, so that $y_S$ is a candidate variable/column for the restricted master problem. Although in principle,  we may assume  that any node of $V$ could be the source, this would produce many symmetric solutions. They are broken imposing that, for any connected $S$, the only source within $S$ is the largest index node.

For this formulation one needs  flow variables $f_{ij}$ defined for all pairs $i,j$ such that $(i,j) \in A$. In addition, the following variables are required.
For $i \in V$, the variable $x_i$ is defined as:
\begin{eqnarray*}
x_i&=&
\begin{cases}
1, & \mbox{if node $i$ is in the cluster,} \\
0, & \mbox{otherwise.}
\end{cases}
\end{eqnarray*}
For any $i, j=1,\ldots,n \mbox{ such that } i< j$, the variable $z_{ij}$ is defined as:
\begin{eqnarray*}
z_{ij}&=&
\begin{cases}
1, & \mbox{if nodes $i$ and $j$ are in the  cluster,} \\
0, & \mbox{otherwise.}
\end{cases}
\end{eqnarray*}

For any $(i,j) \in A$, the variable $f_{ij}$ is defined as:
\begin{eqnarray*}
f_{ij}&=& \mbox{amount of flow sent from node $i$ to node $j$.}
\end{eqnarray*}
The flow-based formulation of \GCSC  is:
\begin{eqnarray}
\label{fo:1}
\mathbf{(F_{flow})} & \min&\displaystyle\sum_{i \in V} \sum_{j \in V: j>i}^nc_{ij}z_{ij}
-\displaystyle\sum_{i=1}^n \gamma_i^* x_i\\
\label{f1:1}
&s.t.&z_{ij}\le x_i, \quad \forall i,j \in V: i<j,\\ 
\label{f1:2}
&& z_{ij}\le x_j, \quad \forall i,j \in V: i<j,\\
\label{f1:3}
& & z_{ij} \ge x_i+x_j-1, \quad \forall i,j \in V: i<j,\\
\label{f2:1}
& &\hspace*{-0.25cm} \sum_{i \in V: (i,k) \in A} f_{ik}- \sum_{i\in V: (k,i) \in A} f_{ki}  \ge x_k+(n-2)( x_j-1),
\,  \forall k,j \in V:  j>k, \\
\label{f2:3}
& & \sum_{j \in V: (i,j) \in A} f_{ij} \le \sum_{j \in V: j <i} z_{ji}+\sum_{j \in V: i<j} z_{ij}, \quad \forall i \in V,\\
\label{z_dom}
& &z_{ij}\ge 0, \quad \forall i, j \in V: i<j,\\
\label{f_dom}
& &f_{ij}\ge 0,  \quad \forall (i,j)\in A,\\
\label{x_dom}
& &x_i\in\{0,1\}, \quad \forall i\in V.
\end{eqnarray}

The objective function \eqref{fo:1}  accounts for the reduced cost. Constraints \eqref{f1:1}-\eqref{f1:3} are the usual inequalities of the Clique Partitioning problem to ensure that $z_{ij}=x_ix_j$.
Constraints \eqref{f2:1} are the flow conservation law, valid for all nodes of
the cluster except for the node with the greatest index. This node is the source, so a flow of the cardinality of the cluster minus one can leave the node.
Constraints \eqref{f2:3} provide an upper bound of the outflow from any node $i \in V$, in addition, if this  node does not belong to the cluster the right hand side of the constraints is 0, i.e., there is not outflow.
Lastly, \eqref{z_dom} - \eqref{x_dom} define the domain of the variables.

An alternative formulation is given in Appendix; where an auxiliary node is considered as source node. That formulation is more natural and intuitive than the one given in this section,  but it provides worse computational results. In spite of that, we decided to keep it in this manuscript because it can ease the understanding of the formulation in this section.

Formulation $\mathbf{(F_{flow})}$ can be strengthened with the families of valid inequalities
 described in \ref{ss:valid-flow} in Appendix.

The minimum reduced cost is $\overline{c}_S=c_S-\sum_{i=1}^n \gamma_i^*x^*_i$, where $c_S=\sum_{i=1}^n\sum_{j=i+1}^n c_{ij}z^*_{ij}$.  If $\overline{c}_S\ge 0$, then the linear relaxation of the master problem is optimal. Otherwise the column $y_S$, that is the incident vector of $S$, is introduced to the restricted master problem (see Step 8 of Algorithm \ref{algo:CG}).

%
%

\subsection{Arborescence formulation  \label{ss:mtzsin}}

The rationale  behind this formulation is that if a node set $S$ is connected, then we can establish a directed spanning subtree using any node of $S$ as the root, and assigning labels to all other nodes of $S$  representing their corresponding positions in the ascending ordered sequence  of  distances from the root to the nodes.
Those type of constraints  are known as Miller-Tucker-Zemlin (MTZ) inequalities,
introduced to solve the Traveling
Salesman Problem in \cite{MTZ60}, and used in other routing problems, \cite{Laporte92, LG1996, BG2014,LM14}.

Let $G_D = (V,A)$ be an auxiliary network as defined in Subsection \ref{ss:flowsin}.
The MTZ description of the Spanning Tree builds an
arborescence rooted at the source node, and in which the arcs follow
the direction from the root to the leaves: Binary variables $t_{ij}$, defined for every $(i,j) \in A$, will take value $1$ if the
arc $(i,j)\in A$ belongs to the arborescence, 0 otherwise. Then, continuous variables $\ell_i$ will indicate the position according to the distance from the root to node $i$ in the ordered sequence of distances from the root to the nodes  using only arcs of the arborescence. Binary variables $x$ and $z$ are defined as in the previous formulation and, as before, to avoid symmetric optimal solutions, for any node set $S$ only the node with the highest index can be the root.


Thus, the arborescence-based formulation of \GCSC is:
\begin{eqnarray}
\nonumber
\mathbf{(F_{MTZ})} &\min& \sum_{i\in V} \sum_{j\in V:i<j} c_{ij}z_{ij}
-\displaystyle \sum_{i \in V} \gamma_i^* x_i \\
\nonumber
&s.t.&  \eqref{f1:1}-\eqref{f1:3}, \eqref{z_dom},\eqref{x_dom} \\
\label{f3:1}
& & \ell_i+1 \le \ell_j+n(1-t_{ij}), \quad \forall (i,j) \in A, \\
\label{f3:2}
& &  t_{ij}+t_{ji}\le z_{ij}, \quad \forall (i,j) \in A: i<j,\\
\label{f4:1}
& &  \sum_{i \in V: (i,k) \in A} t_{ik} \ge x_j+x_k-1, \quad \forall k,j \in V: j>k, \\
\label{f3:7}
& & t_{ij}\in \{0,1\},  \quad \forall (i,j) \in A,\\
\label{f3:10}
& & \ell_i\in \mathbb{R},  \quad \forall i \in V.
\end{eqnarray}

Constraints \eqref{f3:1} guarantee that the label assigned to
node $j$ is at least as great as the label assigned to node $i$ when the arc $(i,j)\in A$ is chosen. Actually, these constraints only avoid cycles, but combined with constraints \eqref{f3:2}, they also exclude
arcs incident to any node $i$ not in $S$.
Constraints \eqref{f4:1} ensure that there is at least one arc incident to all the nodes of $S$ (with the exception of the one with the highest index). Those arcs will form an arborescence, the nodes of the arborescence are the optimal connected component $S$.
Finally, the domain of the variables is defined in \eqref{f3:7} and \eqref{f3:10}.

Formulation $\mathbf{(F_{MTZ})}$ can be strengthened with the family of valid inequalities
 described in \ref{ss:valid-MTZ} of the Appendix.

An alternative formulation, where an auxiliary node is used as source node, is presented in Subsection 
\ref{ss:mtzcon} in the Appendix .  The formulation  $\mathbf{(F_{MTZ})}$ outperforms
that formulation. Nevertheless, we decided to keep it in this manuscript because the former is more natural
and intuitive.

%

\subsection{Relaxations}

If some of the models above are formulated without some   ``model constraints'', then the resulting formulation will be referred to as ``relaxation''. Relaxations are solved faster, but of course, the solution can be unfeasible to the original model. The idea is to iteratively add   constraints to the relaxation, hopefully not too many, from the removed constraints set to find a feasible solution of the original model. Relaxations have been coded in SCIP by implementing a constraint handler (\cite{GleixnerEtal2018OO}). We have explored the possibility of improving the computational times using two relaxations.

\subsubsection{Clique relaxation\label{ss:incomp1}}

In the first relaxation, clique equations $z_{ij}=x_ix_j$,  modeled by constraints \eqref{f1:1}-\eqref{f1:3}, are discarded. Since this type of constraints involves binary variables and they are $O(n^2)$, all MILP problems can be solved faster without their explicit representation. Then, given an incumbent solution, a separation oracle tests  by full enumeration whether it violates some clique inequality and if so, it is inserted into the MILP model. As it will be seen in the computational section, in some cases this strategy has obtained good results.

\subsubsection{Connectivity relaxation\label{ss:incomp2}}

In the second relaxation, connectivity constraints of $\mathbf{F_{flow}}$ and $\mathbf{F_{MTZ}}$ are discarded (while retaining the clique constraints \eqref{f1:1}-\eqref{f1:3}). Suppose that an oracle determines that a node subset $S$ is not connected because at least one pair $i, j \in S$ is not connected in $G[S]$. Then, if $i,j$ should be in the   same node subset, it should include at least one node out of $S$ to be the \textit{bridge} used to connect  $i$ and $j$. That is, for a given $S$ and $i,j \in S$, the connectivity constraints are represented by:
\begin{equation}
\label{expf}
\sum_{\ell=i+1\, : \, \ell \not\in S}^n  z_{i \ell}  +
       \sum_{\ell=1\, : \, \ell \not\in S}^{i-1} z_{\ell i}-z_{ij}\ge 0.
\end{equation}
The formal proof of this result can be found as Theorem 2.1 in \cite{BPR17}.
Note that the number of constraints (\ref{expf}) is exponential, but they can
be separated efficiently.
For a given element of a partition  $S\subseteq V$ of an incumbent solution, consider the auxiliary complete graph $G_S$ in which edge lengths are $l_{ij} = 0$ if $(i,j) \in E[S]$, $l_{ij} = 1$ otherwise. Let $LSP(i,j)$ be the shortest path distance from node $i$ to node $j$ (this can be computed by the Floyd-Warshall algorithm).
If the maximum value of $LSP(i,j)$ for $i,j \in S$ is equal to 1, then subset $S$ is not connected. The formal description of the separation procedure is described in Algorithm \ref{alse}.

\begin{algorithm}[htb]
\KwIn{$G=(V,E)$, $(\bar{z},\bar{x})$ a solution of connectivity relaxation,  $S:= \{k \: : \:  \bar{x}_{k}>0\}$.}
\KwOut{Violated cuts of the family \eqref{expf}.}
\For{$i,j (i<j)\in S$}{
Compute $LSP(i,j)$ in the complete graph $G_S$ with length of edges defined by:
 $$l_{ij}:=
\begin{cases}
0, & \mbox{if $e_{ij}\in E[S],$}\\ 
1, & \mbox{otherwise.}
\end{cases}
$$
\If{$LSP(i,j)>0$ ($i$ and $j$ are not connected in $G[S]$)}{
Add the following inequality of family \eqref{expf}:
\begin{equation}
\label{cutz:1}
 \sum_{\ell=i+1\, : \, \ell \not \in S}^{n}  z_{i \ell}  +
     \sum_{\ell=1\, : \, \ell \not \in S_k}^{i-1} z_{\ell i}-z_{ij}\ge 0.
\end{equation}
}
}
\Return{\hspace*{-0.1cm}\emph{\textbf{: }}All violated cuts found from family \eqref{expf}.}
\caption{Separation Algorithm}
\label{alse}
\end{algorithm}

\section{A shrinking-based and a MILP-relaxed matheuristic\label{sec:4}}

In this section, two new heuristics for the \GCPP  are described.
The first one, called \textit{Random Shrink} (RS) heuristic, is a fast, constructive method to compute quickly an approximate solution. It is flexible enough to be applied to \GCSC problem as well, and in fact it is the heuristic that has been used to solve the pricing problem. The second heuristic is based on the approximated solution of the linear relaxation of (MP)  in each node of the branch-and-bound (B\&B) tree, in which the pricing problems are solved only through the RS heuristic. If the heuristic cannot find a negative reduced cost column, then the algorithm stops.

\subsection{Random shrink heuristic}

\label{sec:3}
This section describes the first new heuristic algorithm devised to solve quickly  both, \GCPP and the \GCSC (with minimal modifications). Finding a feasible solution of the former problem is necessary in Step 1 of Algorithm \ref{algo:CG}, because the master problem must be initialized with a set of variables $y_S$, while solving the latter problem is necessary in Step 6 to find one or more new variables $y_S$ with negative reduced costs. As the algorithm is embedded in a \BP scheme, it must run quickly.

The new algorithm is based on the idea of shrinking the nodes of the graph $G=(V,E)$ in such a way that we  have, in each iteration, a feasible partition, i.e., the elements of the partition are connected subsets as subgraphs on $G$. As a matter of fact, the \GCPP data input is itself a partition, the one in which every singleton is a cluster. If two connected nodes are shrunk, the resulting graph will contain $|V|-1$ nodes, but one node is actually containing two of the original, that is, the partition begins to have a structure. So, shrink can be repeated over and over, until a stopping criterion is satisfied.

More formally, shrink is the operation described in Algorithm \ref{algo:shrink}. Input are the data structure $G^h = <V^h, E^h, c^h, \pi^h>$ and the node pair $i,j \in V^h$ with $e_{ij} \in E^h$, where: $V^h$ is the active node set, each node representing a clique; $E^h$ is the active edge set; $c^h$ are the shrinking costs, defined for every pair $i,j \in V^h$ (when $c^h < 0$ it is actually a gain); $\pi^h_i$ are the clique costs, defined for every active node $i \in V^h$. Furthermore, we define $f(V^h)$ as the objective function of partition $V^h$, $f(V^h) =\sum_{i \in V^h} \pi^h_i$.

The output is a data structure $G^{h+1} = <V^{h+1}, E^{h+1}, c^{h+1}, \pi^{h+1}>$, in which $|V^{h+1}| = |V^h|-1$. When pair $i,j \in V^{h}$ is shrunk, then $j$ and $(i,j)$ are deleted from nodes and edges respectively. Then, the clique costs $\pi^h_i$  increases or decreases by cost $c^h_{ij}$, see Steps 3 and 4. All links and the costs of $j$ are allocated to $i$, see Steps 5-8. Finally, the objective function $f(V^h)$ of the \GCPP is updated in Step 9. Note that if we have to solve the \GCSC, then in Step 9 we can define
$f(V^{h+1}) =\min_{i \in V^{h+1}} \{\pi^{h+1}_i\}$.

\begin{algorithm}[htb]
\DontPrintSemicolon
\KwIn{Data structure: $G^h = <V^h, E^h, c^h, \pi^h>$, the pair $i,j\in V^h$ with $e_{ij} \in E^h$.}
\KwOut{The data structure: $G^{h+1} = <V^{h+1}, E^{h+1}, c^{h+1}, \pi^{h+1}>$.}
$V^{h+1} \gets V^h  \setminus \{j\}$\;
$E^{h+1} \gets E^h  \setminus \{ e_{ij} \} $\;
$\pi^{h+1}_k \gets \pi^h_k \: \forall k (\ne i,j) \in V^h$\;
$\pi^{h+1}_i \gets \pi^h_i + c^h_{ij}$\;
$c_{k \ell}^{h+1} \gets c^h_{k\ell} \:\: \forall e_{k \ell} \in E^{h+1}$\;
       \For{$k \in V^h: e_{jk} \in E^h$}{
				$E^{h+1} \gets E^{h+1} \cup$ $ (i,k) - (j,k)$ \;
				$c^{h+1}_{ik} \gets c^{h+1}_{ik} + c^h_{jk}$\;
		}
$f(V^{h+1}) \gets$ $\sum_{k \in V^{h+1}} \pi^{h+1}_k$\;
\Return{$G^{h+1}$}\;
\caption{ Subroutine SHRINK}
\label{algo:shrink}
\end{algorithm}

Before applying subroutine SHRINK in Algorithm  \ref{algo:shrink},  an edge $e_{ij} \in E^h$ must be elicited, but then, the choice can favor optimality or diversification.
According to the \textit{optimality criterion}, $i$ and $j$ must be such that the cost $c_{ij}^h$ is minimum. In this way, if the cost is negative, shrinking $i,j$ is the best decrease of the incumbent objective function $f(V^h)$. According to the \textit{diversification criterion}, $i$ and $j$ can be selected randomly, but preferably the pair has been often assigned to different clusters in previous local optima.

\begin{algorithm}[htb]
\DontPrintSemicolon
\KwIn{The \GCPP problem, max\_start,  max\_random\_move.}
\KwOut{The optimal partition: $G^{best}$.}
\For{$s := 1 \, \mathbf{to}  \mbox{ max\_start}$}{
\If{ $s > 1$}{
	$\mbox{random\_move = Unif(1,max\_random\_move)}$\;
	\For{$t := 1 \, \mathbf{to} \mbox{ random\_move}$}{
		 $e_{ij}$ $\gets \mbox{Random\_choice}(W)$\;}
		 $G^{h+1} \gets \mbox{Shrink}(G^h)$\;
		 $h \gets q$\;	
	}
$\mbox{fine := false}$\;
\While{$\mbox{fine = false}$}{
$e_{ij}$ $\gets \arg \min \{c_{ij}^h\}$\;
\If{$c_{ij}^h < 0$}{
$G^{h+1} \gets \mbox{Shrink}(G^h)$\; 	
$h \gets q$\;
}
\Else{
$G^{best} \gets \mbox{Update\_Best}(G^h)$\;
$W \gets \mbox{Update\_Weight}(W)$\;
$\mbox{fine := true}$\;
} 
} 
} 
\Return{$G^{best}$}\;
\caption{{\sc Random Shrink}}
\label{algo:rs}
\end{algorithm}

The \textit{Random Shrink} (RS) procedure is described in Algorithm \ref{algo:rs}. Input data are an instance of \GCPP, and parameters: $max\_start$ and  $max\_random\_move$.
At the beginning, every cluster is a singleton:
$V^h = V$, $E^h = E$, $\pi^h = 0$, $f(V^h) = 0$. Then the graph is shrunk until a local optimum is found. In the first run, the method is greedy: Random moves are skipped, see Step 2. From the second round onwards, the first selections of pairs $i,j$, such that $e_{ij}\in E^{h}$, are random, see Steps 4-7. The number of random moves is itself random (drawn from a discrete uniform distribution from 1 to $max\_random\_move$), and depends on the input parameter $max\_random\_move$, see Step 3.

The loop of Steps 9-17 is a standard greedy procedure, in which the best edge $e_{ij}$ is selected in Step 10. The graph is shrunk if it provides an improvement of the objective function (Steps 11-13), otherwise, if necessary, the procedure updates the best solution so far (Steps 15-17). All is iterated $max\_start$ times, an input parameter, see  Step 1. In every iteration, information about all local optima is stored in matrix $W$. The role of $W$ is to lead the diversification: When implementing the random choice of $e_{ij}$, it is taken into account how many times an edge $e_{ij}$ has been in local optima (that is, $i$ and $j$ were put into different clusters). The most it has been excluded from local optima, the highest is the probability of being selected randomly. To this purpose, when an egde $e_{ij}$ is not in the local optimum $E^{h}$, then the value $w_{ij}$  is augmented by one. When doing a random choice, the probability of choosing $e_{ij}$ is $Pr[i,j] = w_{ij}/W$, with $W =\sum_{e_{ij} \in E} w_{ij}$.

\subsection{A new MILP-relaxed matheuristic}
\label{ss:matheu}

The \BP described in Algorithm \ref{algo:CG} can be readily modified to calculate an approximate solution instead of the optimum. It is sufficient to solve the pricing problems using only the RS heuristic, and never calculate the exact solution of the different \GCSC problems. Branching is still done to solve the master problem \GCPP, as it is usually not much time consuming. In other words, (MP) is solved adding columns which empirically are tested to be useful, but not enough to certify optimality. In this way, \GCPP is heuristically solved very quickly, but at the price of only solving approximately each linear relaxation of the master problem at any node of the B\&B tree. In spite of that, as our computational results show, the quality of the solutions are rather good.

\section{A branch-and-price implementation\label{sec:5}}

In this section, we describe technical details of Algorithm \ref{algo:CG}, that were set aside so far for the sake of brevity. They are the generation of an initial solution, the branching rule, the Farkas pricing, and the convergence of column generation.

\subsection{Starting solutions}
\label{pricerheu}

Good starting solutions, that is, the initial clusters $y_S$'s with their costs, are important to prune the searching tree. So, in this phase the RS algorithm is run with an abundant iteration limit and all local optima are used to define initial variables $y_S$'s and feasible solutions of \GCPP.

\subsection{Ryan-and-Foster branching}

Branching occurs when the LP solution of the master problem contains fractional variables. In B\&P, it is not trivial to define a branching rule to resolve fractional solutions without fixing variables that were already in the pool of columns, \cite{Barnhart96branch-and-price:column}. Here, in Step 19 of Algorithm \ref{algo:CG}, the Ryan-and-Foster branching has been implemented, as it has considerable advantages over alternatives.

The Ryan-and-Foster (R-F) has been introduced to solve set partitioning problems, see \cite{Ryan1981}, and now is one of the most widespread techniques. If at a node of the master problem a solution contains fractional variables, the R-F rule creates two new branches as follows: Given two elements $i_1,i_2\in V$, in one branch they will always be in the same cluster, whereas in the other branch they will always be in different cluster.

To implement this branching, we can take advantage of the $x_i$ variables defined on the previous section for the pricing subproblem:

\begin{itemize}
\item {\bf Left branch: } If $i_1$ and $i_2$ must be in different clusters implies that none of the variables corresponding to clusters containing $i_1$ and $i_2$ can assume positive values, i.e.,
$$\sum_{S\ni i_1,i_2}y_S=0 \Rightarrow x_{i_1}+x_{i_2}\le 1.$$
\item {\bf Right branch: } Since $i_1$ and $i_2$ must be in the same cluster then the following sum must be equal to 1:
$$\sum_{S\ni i_1,i_2}y_S=1 \Rightarrow x_{i_1}=x_{i_2}.$$
\end{itemize}

In practice, when a new node is created (or candidate to be solved), existing $y_S$ variables local bounds are modified according to the above constraints. These bounds are taken into account in Step 5 of Algorithm \ref{algo:CG} when function $\mbox{Solve}\mathbf{((Relaxed MP )}_{\mathbb{S}},\mbox{node})$ is called.

Furthermore, to solve the pricing problem, new variables not satisfying node requirements should be avoid. In function $\mbox{Solve\_Pricing\_Problem}( \gamma^*,\mbox{node})$ (Step 6 of Algorithm \ref{algo:CG}) we include the information of the ancestor nodes: $x_{i_1}+x_{i_2}\le 1$ (left branch); and $x_{i_1}=x_{i_2}$ (right branch).

Finally, Steps 9--19 of Algorithm \ref{algo:CG} work similar to the common B\&B algorithm with some particularities of our R-F branching. When a fractional solution is found Branch($y^*$) finds a pair $i_1,i_2\in I$ for which

$$0<\sum_{S\ni i_1,i_2}y_S<1,$$

\noindent to create left and right nodes, using most fractional criterion. Lower\_Bound($y^*$)  and Upper\_Bound($y^*$) update the lower  and upper bound of (MP), respectively. Both functions return TRUE in case that bounds coincide, so (MP) is solved. Otherwise, Next\_Node(MP) decides which is the next node to be studied.  We let solver (SCIP,  \cite{GleixnerEtal2018OO}) default-settings decide on the next node to be explored.

\subsection{Farkas pricing}

Another important element in any \BP algorithm is the so called Farkas pricing. This is the subroutine that provides new columns to the restricted master if it is locally infeasible.

We observe that infeasibility only can happen on a new node of the branching tree. If it happens because the R-F branching produces incompatible conditions, then the node is declared infeasible and no call to any pricing problem is necessary. Otherwise, the R-F conditions are compatible but perhaps not enough $y_S$ variables are available in the pool to build a feasible solution. However, one can ensure \textit{fictitious} feasible solutions by the following construction.
\begin{propo}
Assume that one initializes the pool of columns with all the singletons $y_{\{i\}}$ for all $i\in V$ and all pairs $y_{\{i,j\}}$ for all $i,j\in V$. Then, if the R-F branching leads to a node with compatible conditions \GCPP is always feasible.
\end{propo}
\begin{proof}
The reader may note that if $e_{ij}\not \in E$, we are augmenting in the initial pool a fictitious edge to $E$ with cost $\hat c_{ij}=+M, \quad M\gg 0$. 
These variables always ensure \textit{fictitious} feasible solutions of the restricted master problem (actually they may not be connected). Moreover, if it happens that in a node, one of those fictitious elements is used in a partition, it would represent an actual infeasible solution but it will never be optimal. 
$\hfill{\Box}$
\end{proof}

In conclusion, the above result justifies that \GCPP does not need a Farkas pricing routine.

\subsection{Convergence of column generation}

In column generation, it is well-known that the columns which certify optimality emerge at the last iterations of the procedure. This phenomenon has been studied and different solutions have been proposed in the literature to overcome it. Among others, \cite{DUMERLE1999229}, \cite{Pessoa2010}, and \cite{SATO2012636}  have designed procedures to minimize the negative impact of the issue in the convergence of column generation algorithms. See also \cite{SATO2019236} or \cite{Deleplanque2020} for other recents applications of those techniques.

Those stabilization procedures are based on the principle that adding in each iteration the column with the best reduced cost may lead to convergence problems. In some way the conclusion of those papers is that  the pricing problem optimal solution should be taken into account only in the latter iterations. Following that principle and basing on the results of Section \ref{section:6.1}, we solve the pricing problem heuristically for the first iterations. Hence, our algorithm stabilizes itself  (\cite{BLANCO2021105124}) as it is supported with our empirical results shown in Table \ref{answer:table1}.

\section{Computational studies\label{s:comp-test}}

Algorithms are tested on the instances previously used in \cite{BPR17}, and on new instances with greater size. The experiment layout is as proposed in \cite{NAD-2003}:
Data are composed of $n$ units on which $m$ binary features, $F_i = \{0,1\}, i = 1,\ldots,m$, are recorded. Units belong to one of two groups, each group is composed of $n/2$ units. If one unit belongs to group 1, then  $\Pr[F_i = 1] = p_c$ for all $i = 1,\ldots,m$, otherwise, if the unit belongs to group 2, then $\Pr[F_i = 1] = 1 - p_c$ for all $i = 1,\ldots,m$. If $p_c$ is close to one, then the two groups are well separated, as $p_c$ gets closer to 0.5, the separation is less and less precise. Units are connected through arcs: If two units (or nodes) belong to the same group, then the probability of a joining arc is $p_{in}$ (the probability of an inner arc). If the two nodes belong to two different groups, then the probability of a joining arc is $p_{out}$ (the probability of an outer arc). For the effect of the probabilities, the number $X_{in}$ of vertices of the same group and the number $X_{out}$ of vertices of the other group to which a given vertex $i \in V$ is connected are two random variables, with expected values $E[X_{in}] \approx n p_{in}/2$ and  $E[X_{out}] \approx n p_{out}/2$. All experiments are run with $p_{in} > p_{out}$, so that connectivity provides information: If a node $i$, whose membership is uncertain, is connected with a node $j$ that is known to belong to Group $k$, then it is likely that $i$ belongs to $k$ as well.

In all our computational experience, models are coded in C and solved with SCIP 6.0.1, \cite{GleixnerEtal2018OO}, using the optimization solver CPLEX 12.8  on an Intel(R) Core(TM) i7-4790 CPU  @4.00 GHz  32GB RAM.
SCIP is a C library of subroutines specially devised to implement branch-cut-and-price and is distributed free-of-charge, under academic license, by the Zuse Institute  Berlin (ZIB). We thank the SCIP team for the helpful technical advices in the course of this research.

\subsection{Deciding the pricing problem implementation}\label{section:6.1}
 From now on, we call heuristic pricer the application of any heuristic for solving the pricing problem. In case that the pricing problem is optimally solved, we call it exact pricer. First of all, we want to decide whether combining exact and heuristic pricers is worth. We have begun by analyzing the performance of Algorithm \ref{algo:CG} for solving to optimality \GCPP. For that reason, to test the usefulness of combining the heuristic and the exact pricers, we run a pilot study on instances of size 20, 30 and 36 nodes. We have compared two different implementations: One  combining heuristic and the exact pricers, the other one only using the exact pricer. In addition, two version of exact pricers have been tested as well, one using the Flow-based formulation, see Subsection \ref{ss:flowsin}, and the other using the Arborescence formulation, see Subsection \ref{ss:mtzsin}. Models $\mathbf{(F^0_{flow})}$ and $\mathbf{(F^0_{MTZ})}$, see Appendix \ref{appendixA}, were discarded at an early stage of our computational experiments since preliminary results show that the use of the auxiliary node does not add any advantage concerning computational time.

Figure \ref{performanceProfiles-heu} reports the results of the 60 instances tested for each implementation (three sizes, ten instances per size, and two formulations). It compares  the number of solved instances versus time of Algorithm \ref{algo:CG} using the Flow-based and the Arborescence formulations, and combining or excluding the heuristic pricer. One can observe that the combination of the exact and the heuristic pricer (line Heurvar=TRUE) is better than excluding the heuristic pricer (Heurvar=FALSE). These results suggest that solving the pricer problem combining the RS heuristic and any exact MILP is more efficient than using only MILP. Therefore, this is the strategy implemented to the largest instances too.

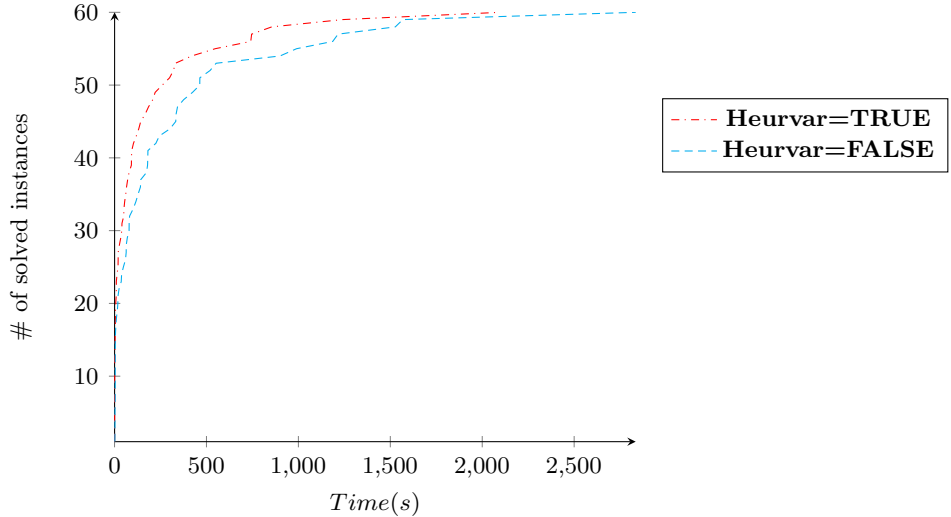
\begin{figure}[htb] \centering
	\begin{tikzpicture}[scale=1.0,font=\footnotesize]
	\begin{axis}[axis x line=bottom,  axis y line=left,
	xlabel=$Time(s)$,
	ylabel=\# of solved instances,
	legend style={at={(1.6,0.8)}}]
	
	\addplot[red,dashdotted] plot coordinates {		

(	0.33	,	1	)
(	0.39	,	2	)
(	0.47	,	3	)
(	0.63	,	4	)
(	0.68	,	5	)
(	0.71	,	6	)
(	0.8	,	7	)
(	0.83	,	8	)
(	0.92	,	9	)
(	1.1	,	10	)
(	1.2	,	11	)
(	1.31	,	12	)
(	1.99	,	13	)
(	2.26	,	14	)
(	3.37	,	15	)
(	3.52	,	16	)
(	4.47	,	17	)
(	5	,	18	)
(	5.07	,	19	)
(	6.41	,	20	)
(	9	,	21	)
(	9.74	,	22	)
(	9.8	,	23	)
(	13.57	,	24	)
(	19.01	,	25	)
(	19.620001	,	26	)
(	19.700001	,	27	)
(	26.190001	,	28	)
(	33.759998	,	29	)
(	39.389999	,	30	)
(	40.439999	,	31	)
(	49.91	,	32	)
(	51.09	,	33	)
(	55.150002	,	34	)
(	60.939999	,	35	)
(	65.269997	,	36	)
(	71.330002	,	37	)
(	79.089996	,	38	)
(	91.699997	,	39	)
(	92.010002	,	40	)
(	93.519997	,	41	)
(	102.970001	,	42	)
(	118.330002	,	43	)
(	131.839996	,	44	)
(	142.369995	,	45	)
(	166.490005	,	46	)
(	184.669998	,	47	)
(	208.009995	,	48	)
(	220.899994	,	49	)
(	259.609985	,	50	)
(	299.410004	,	51	)
(	320.100006	,	52	)
(	331.140015	,	53	)
(	419.75	,	54	)
(	547.619995	,	55	)
(	740.960022	,	56	)
(	745.700012	,	57	)
(	855.150024	,	58	)
(	1240.26001	,	59	)
(	2097.23999	,	60	)

	};
	\addlegendentry{\textbf{Heurvar=TRUE}}	
	\addplot[cyan, densely dashed] plot coordinates {		
(	1.36	,	1	)
(	1.98	,	2	)
(	2.19	,	3	)
(	2.29	,	4	)
(	2.37	,	5	)
(	2.82	,	6	)
(	3.44	,	7	)
(	3.56	,	8	)
(	3.62	,	9	)
(	3.66	,	10	)
(	3.66	,	11	)
(	3.82	,	12	)
(	3.91	,	13	)
(	3.99	,	14	)
(	4.05	,	15	)
(	5.42	,	16	)
(	7.15	,	17	)
(	10.35	,	18	)
(	15.26	,	19	)
(	18.1	,	20	)
(	18.76	,	21	)
(	24.51	,	22	)
(	36.459999	,	23	)
(	38.419998	,	24	)
(	48.09	,	25	)
(	59.470001	,	26	)
(	63.32	,	27	)
(	63.619999	,	28	)
(	69.300003	,	29	)
(	79.110001	,	30	)
(	79.43	,	31	)
(	81.93	,	32	)
(	101.559998	,	33	)
(	116.279999	,	34	)
(	127.959999	,	35	)
(	140.910004	,	36	)
(	142.419998	,	37	)
(	173.960007	,	38	)
(	179.639999	,	39	)
(	180.850006	,	40	)
(	181.589996	,	41	)
(	224.929993	,	42	)
(	243.339996	,	43	)
(	301.769989	,	44	)
(	332.950012	,	45	)
(	334.75	,	46	)
(	342.350006	,	47	)
(	376.089996	,	48	)
(	424.160004	,	49	)
(	463.589996	,	50	)
(	464.040009	,	51	)
(	519.559998	,	52	)
(	552.580017	,	53	)
(	900.119995	,	54	)
(	990.219971	,	55	)
(	1186.849976	,	56	)
(	1215.75	,	57	)
(	1527.380005	,	58	)
(	1568.089966	,	59	)
(	2834.840088	,	60	)	
	};
	\addlegendentry{\textbf{Heurvar=FALSE}}	
	\end{axis}
	\end{tikzpicture}
	\caption{Performance profile graph of \#solved instances using the combined heuristic and  exact pricers  or only using the exact pricer for $n=$20,30,36 (the exact pricer uses two formulations: $\mathbf{(F_{flow})}$ and $\mathbf{(F_{MTZ})}$).}  \label{performanceProfiles-heu}
\end{figure}

Concerning the tailing off effect of this heuristic pricer, Table \ref{answer:table1} shows the number of necessary variables to certify optimality for instances of different size, depending on whether heuristic pricer is applied (\emph{TRUE}) or not (\emph{FALSE}). In this table, \emph{Initial} is the average number of variables added from the beginning, \emph{Heur} is the average number of variables added after the heuristic pricer interation, and \emph{Exact} is the average number of variables added when the pricing problem is solved exactly.
\begin{table}[htb]
\scriptsize
\centering
\begin{tabular}{|l|rrrr|rrrr|rrrr|}
\hline
	& \multicolumn{4}{|c|}{$n=20$} & \multicolumn{4}{|c|}{$n=30$} & \multicolumn{4}{|c|}{$n=36$}\\
\hline
 Heurvar& Initial & Heur & Exact & \bf Total& Initial & Heur & Exact & \bf Total& Initial & Heur & Exact & \bf Total\\
 \hline
 FALSE&51.5&0.0&27.6&\bf79.1&79.4&0.0&69.5&\bf 148.9&98.3&0.0&164.9&\bf 263.2\\
 TRUE&51.5&23.5&5.4&\bf80.4&79.4&43.7&12.8&\bf 135.9&98.3&102.8&34.3&\bf 235.4\\
  \hline
 Variation&&&&\bf +1.7\%&&&& \bf -8.8\%&&&& \bf-10.6\% \\
 \hline

\end{tabular}
\caption{Average number of variables using the combined heuristic and  exact pricers  or only using the exact pricer for $n=$20,30,36 (the exact pricer uses two formulations: $\mathbf{(F_{flow})}$ and $\mathbf{(F_{MTZ})}$).}
\label{answer:table1}
\end{table}

When the heuristic pricer is applied, the problem is solved using a smaller number of variables. It means that the pricer heuristic not only saves computational time but also reduces degeneracy (that is, the situation in which reduced cost variables do not decrease the objective function). Furthermore, it can be seen that the impact is more remarkable for bigger instances.

\subsection{Comparison of different formulations of \GCSC}
We have continued our study solving (MP) using different alternatives for the pricing problem subroutine. In this set of experiments, as recommended by the above pilot study, Algorithm \ref{algo:CG} has been run combining  the RS heuristic, and using the different MILP
formulations (see Section \ref{sec:pricing}) to solve the pricing problem. 

Five  different MILP pricing routines are compared.
The first two MILP models were $\mathbf{(F_{flow})}$ and $\mathbf{(F_{MTZ})}.$
In the next three formulations, MILP are initialized without some constraints and/or variables, that are included whenever necessary to separate infeasible solutions only after that the separation subroutine is invoked. We refer to Flow Clique Relaxation and MTZ Clique Relaxation when clique constraints are removed (see Subsection \ref{ss:incomp1}) from Flow-based and Arborescence formulations, respectively. The fifth formulation, Connectivity Relaxation, removes the connectivity constraints (see Subsection \ref{ss:incomp2}).

The results reported in Table \ref{t:Form05} are averages calculated after solving ten instances of each size, letting a maximum of 24 hours of computation.
This table contains five blocks, one for each implementation of Algorithm \ref{algo:CG}. We report there the average solution time (Av.Time), the average gap at termination (Av.GAP), 
and the number of unsolved instances after the time limit is reached  (Unsol). The best results in terms of times, gaps and number of unsolved problems are written in bold.

\begin{rem}
The lower bound used to calculate the gap at termination is given by the B\&B process as usually. However, if the linear relaxation of the MP has not been solved at the time limit, another lower bound is still available, see \cite{LubbeckeMarcoE2005STiC}. Particularly, for the GCCP the lower bound during the resolution of the root node is

$$LB= \sum_{i=1}^n \gamma_i^*+n\min_{S\in \mathcal{S}}\bar c(y_S),$$

\noindent provided that the last pricing problem has been solved exactly.
\end{rem}

\begin{table}[htb]
\centering
\begin{adjustbox}{max width=1.0\textwidth}
\begin{tabular}{|c|r|r|r|r|r|r|r|r|r|r|r|r|r|r|r|}
\hline
	Model	& \multicolumn{3}{c|}{$\mathbf{(F_{flow})}$} & \multicolumn{3}{c|}{$\mathbf{(F_{MTZ})}$} & \multicolumn{3}{c|}{Flow Clique Relaxation} &	\multicolumn{3}{c|}{MTZ Clique Relaxation} &	\multicolumn{3}{c|}{Connectivity Relaxation}\\	 \hline	
	n	&	Av.Time	&	Av.GAP	& Unsol	&	Av.Time	&	Av.GAP	&	Unsol &	Av.Time	&	Av.GAP	&	Unsol &	Av.Time	&	Av.GAP	&	Unsol &	Av.Time	&	Av.GAP	&	Unsol	\\ \hline
20	&	1.39	&\bf	0.00	&\bf	0	&	4.74	&\bf	0.00	&\bf	0	&	2.14	&\bf	0.00	&\bf	0	&	2.16	&\bf	0.00	&\bf	0	&\bf	1.25	&\bf	0.00	&\bf	0	\\
30	&\bf	34.71	&\bf	0.00	&\bf	0	&	64.36	&\bf	0.00	&\bf	0	&	62.49	&\bf	0.00	&\bf	0	&	40.75	&\bf	0.00	&\bf	0	&	53.37	&\bf	0.00	&\bf	0	\\
36	&\bf	419.81	&\bf	0.00	&\bf	0	&	547.49	&\bf	0.00	&\bf	0	&	546.17	&\bf	0.00	&\bf	0	&	739.49	&\bf	0.00	&\bf	0	&	817.11	&\bf	0.00	&\bf	0	\\
40	&	3545.18	&\bf	0.00	&\bf	0	&	1503.33	&\bf	0.00	&\bf	0	&	2731.25	&\bf	0.00	&\bf	0	&\bf	1444.89	&\bf	0.00	&\bf	0	&	5319.05	&\bf	0.00	&\bf	0	\\
50	&	18331.38	&\bf  0.00	&\bf	0	&	24820.23	&	14.03	&	2	&\bf	16634.83	&\bf	0.00	&\bf	0	&	21320.21	&	1.49	&	2	&	40006.47	&	0.01	&	1	\\
54	&	50646.86	&	1.03	&\bf	3	&\bf	43684.52	&	5.01	&\bf	3	&	51861.38	&	1.30	&\bf	3	&	44095.21	&\bf	0.66	& \bf	3	&	71221.69	&	4.12	&	7	\\
60	&	81152.89	&	1.83	&	5	&	60697.51	&	28.28	&	6	&	80950.89	&	3.74	&	7	&\bf	56988.24	&	\bf1.96	&\bf	4	&	86405.08	&	6.68	&	10	\\ \hline
Total Result	&\bf	22018.89	&\bf	0.41	&\bf	8	&\bf	18760.31	&\bf	6.76	&\bf	11	&\bf	21391.74	&\bf	0.71	&\bf	10	&\bf	17423.4	&\bf	0.59	&\bf	9	&\bf	29117.72	&\bf	1.54	&\bf	18	\\

\hline
\end{tabular}
\end{adjustbox}
\caption{Average results for models with  pricing problems based on formulations introduced in Section \ref{sec:pricing}. \label{t:Form05}}
\end{table}

The results in this table point out that the best formulations are  $\mathbf{(F_{flow})}$ and MTZ Clique Relaxation. The former solves 62 out 70 instances up to optimality, the latter solves 61 instances, that is, one problem less, but with a slightly smaller average computational time. Comparing Algorithm \ref{algo:CG} with previous MILP methods, reported in \cite{BPR17}, we can observe that the maximum solved size has been improved from 40 to 60 units with the same time limit, and that computational times for solved instances have improved as well.

In Figure \ref{performanceProfiles-all}, the results of Table \ref{t:Form05} are summarized.
Profiles show that pricing routines based on $\mathbf{(F_{MTZ})}$ and MTZ 
Clique Relaxation formulations are giving the best performance in terms of number of solved instances. However, Arborescence formulations $\mathbf{(F_{MTZ})}$ and the (MTZ Clique Relaxation) are giving the best solution times when the instances are solved.
It can also be seen that removing the connectivity constraints (Connectivity Relaxation) does not work better than removing the clique  constraints (Flow Clique Relaxation and MTZ Clique Relaxation). The latter do not solve 10 and 9 instances, respectively, and the former 18 instances.

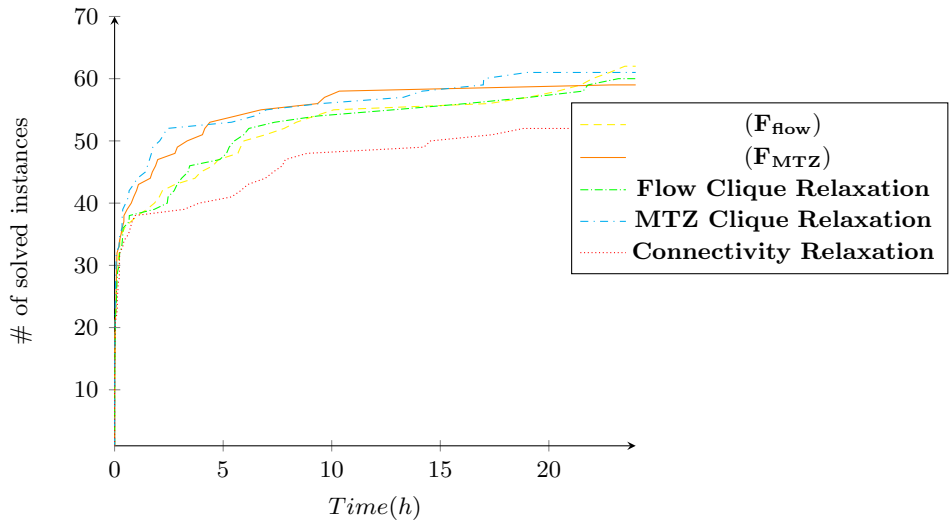
\begin{figure}[htb] \centering
	\begin{tikzpicture}[scale=1.0,font=\footnotesize]
	\begin{axis}[axis x line=bottom,  axis y line=left, ytick={10,20,30,40,50,60,70},ymax=70,
	xlabel=$Time(h)$,
	ylabel=\# of solved instances,
	legend style={at={(1.6,0.8)}}]

\addplot[yellow,densely dashed] plot coordinates {
(	0.0000667636111111111	,	1	)
(	0.000185567222222222	,	2	)
(	0.000204615277777778	,	3	)
(	0.000206085277777778	,	4	)
(	0.000233699444444444	,	5	)
(	0.000272559166666667	,	6	)
(	0.000286311944444444	,	7	)
(	0.000323341388888889	,	8	)
(	0.000861319166666667	,	9	)
(	0.00120986	,	10	)
(	0.00325020722222222	,	11	)
(	0.00445815777777778	,	12	)
(	0.00489096472222222	,	13	)
(	0.00508665027777778	,	14	)
(	0.00799159694444445	,	15	)
(	0.0101240011111111	,	16	)
(	0.0110942872222222	,	17	)
(	0.0137459977777778	,	18	)
(	0.0155619811111111	,	19	)
(	0.0202033191666667	,	20	)
(	0.0237932077777778	,	21	)
(	0.051750145	,	22	)
(	0.0543435838888889	,	23	)
(	0.0734329733333333	,	24	)
(	0.0747165933333333	,	25	)
(	0.0769417063888889	,	26	)
(	0.0829288058333333	,	27	)
(	0.0836097716666667	,	28	)
(	0.09258205	,	29	)
(	0.1140784625	,	30	)
(	0.172963324722222	,	31	)
(	0.200975324722222	,	32	)
(	0.234750942777778	,	33	)
(	0.248973371666667	,	34	)
(	0.320278015	,	35	)
(	0.461107313333333	,	36	)
(	0.734388020833333	,	37	)
(	1.06653225361111	,	38	)
(	1.48282009555556	,	39	)
(	1.83418972444444	,	40	)
(	2.05990519194444	,	41	)
(	2.20031480583333	,	42	)
(	2.78140597861111	,	43	)
(	3.70168212888889	,	44	)
(	3.96346489805556	,	45	)
(	4.52633382166667	,	46	)
(	4.78580620666667	,	47	)
(	5.67710449222222	,	48	)
(	5.80103569888889	,	49	)
(	5.95653591583333	,	50	)
(	6.95638020833333	,	51	)
(	7.8043733725	,	52	)
(	8.37323676222222	,	53	)
(	9.21995225694444	,	54	)
(	10.0933886719444	,	55	)
(	17.2338964844444	,	56	)
(	19.0393706597222	,	57	)
(	20.4126453994444	,	58	)
(	21.4725412327778	,	59	)
(	22	,	60	)%
(	22.8	,	61	)%
(	23.5	,	62	)%
(	24	,	62	)

	};
	\addlegendentry{\textbf{$\mathbf{(F_{flow})}$}}

\addplot[orange,solid] plot coordinates {
(	0.000178637777777778	,	1	)
(	0.0002030075	,	2	)
(	0.000290795277777778	,	3	)
(	0.000305183333333333	,	4	)
(	0.0006167975	,	5	)
(	0.000639271388888889	,	6	)
(	0.000930536666666667	,	7	)
(	0.00201857888888889	,	8	)
(	0.00204510694444444	,	9	)
(	0.002923325	,	10	)
(	0.00369192888888889	,	11	)
(	0.00378533916666667	,	12	)
(	0.0041599225	,	13	)
(	0.00847580805555556	,	14	)
(	0.0132430194444444	,	15	)
(	0.0205707125	,	16	)
(	0.0207017369444444	,	17	)
(	0.0220450761111111	,	18	)
(	0.0249273469444444	,	19	)
(	0.030378225	,	20	)
(	0.0332247947222222	,	21	)
(	0.0492453638888889	,	22	)
(	0.0505171880555556	,	23	)
(	0.0591862275	,	24	)
(	0.0705013233333333	,	25	)
(	0.0737217033333333	,	26	)
(	0.0813317872222222	,	27	)
(	0.0848756577777778	,	28	)
(	0.0926170941666667	,	29	)
(	0.0961183080555556	,	30	)
(	0.101141484444444	,	31	)
(	0.11086007	,	32	)
(	0.200856441944444	,	33	)
(	0.227349090555556	,	34	)
(	0.296521911666667	,	35	)
(	0.346077406111111	,	36	)
(	0.426826171944444	,	37	)
(	0.430642191666667	,	38	)
(	0.5780992975	,	39	)
(	0.767091471388889	,	40	)
(	0.879252861944445	,	41	)
(	1.0083936225	,	42	)
(	1.09828430166667	,	43	)
(	1.64253377277778	,	44	)
(	1.73704264333333	,	45	)
(	1.8865437825	,	46	)
(	1.98610582138889	,	47	)
(	2.78002983944444	,	48	)
(	2.88890190972222	,	49	)
(	3.34180501305556	,	50	)
(	4.04686903222222	,	51	)
(	4.15389322916667	,	52	)
(	4.38847629111111	,	53	)
(	5.57890625	,	54	)
(	6.74795789944445	,	55	)
(	9.3509375	,	56	)
(	9.66791883666667	,	57	)
(	10.3490017361111	,	58	)
(	22.8946397569444	,	59	)
(	24	,	59	)

	};
	\addlegendentry{\textbf{ $\mathbf{(F_{MTZ})}$}}	
	\addplot[green,densely dashdotted] plot coordinates {
(	0.000105298888888889	,	1	)
(	0.000197283611111111	,	2	)
(	0.000280831666666667	,	3	)
(	0.000316824722222222	,	4	)
(	0.000339676388888889	,	5	)
(	0.000401379444444444	,	6	)
(	0.000411512222222222	,	7	)
(	0.000623387777777778	,	8	)
(	0.00140996388888889	,	9	)
(	0.00186624527777778	,	10	)
(	0.00355899972222222	,	11	)
(	0.00431783805555556	,	12	)
(	0.00626052333333333	,	13	)
(	0.0084328275	,	14	)
(	0.0136045222222222	,	15	)
(	0.0154468344444444	,	16	)
(	0.0216386666666667	,	17	)
(	0.0237645466666667	,	18	)
(	0.0337886788888889	,	19	)
(	0.03738753	,	20	)
(	0.037542055	,	21	)
(	0.0390270361111111	,	22	)
(	0.0672827997222222	,	23	)
(	0.0782807244444444	,	24	)
(	0.0876675669444445	,	25	)
(	0.0892197419444445	,	26	)
(	0.109881320555556	,	27	)
(	0.122149505555556	,	28	)
(	0.138667085555556	,	29	)
(	0.210419684444444	,	30	)
(	0.210453457222222	,	31	)
(	0.236364678333333	,	32	)
(	0.327780253055555	,	33	)
(	0.357161390555556	,	34	)
(	0.377403428888889	,	35	)
(	0.391169603055556	,	36	)
(	0.663210856111111	,	37	)
(	0.677992011111111	,	38	)
(	1.82673773861111	,	39	)
(	2.42485785583333	,	40	)
(	2.46278591583333	,	41	)
(	2.74378634972222	,	42	)
(	2.87855333111111	,	43	)
(	3.08573730472222	,	44	)
(	3.37824408638889	,	45	)
(	3.45863444	,	46	)
(	4.84395779083333	,	47	)
(	5.15554307722222	,	48	)
(	5.27684733083333	,	49	)
(	5.51929361972222	,	50	)
(	5.95235731333333	,	51	)
(	6.18081759972222	,	52	)
(	7.37353895388889	,	53	)
(	9.34189236111111	,	54	)
(	12.8449359808333	,	55	)
(	16.1934982638889	,	56	)
(	19.0919661458333	,	57	)
(	21.5766579861111	,	58	)
(	21.8	,	59	)%
(	23.2	,	60	)%
(	24	,	60	)

	};
	\addlegendentry{\textbf{Flow Clique Relaxation}}		
	
	\addplot[cyan,dashdotted] plot coordinates {		
	
(	0.0000562447222222222	,	1	)
(	0.0000617547222222222	,	2	)
(	0.000112851666666667	,	3	)
(	0.000114038055555556	,	4	)
(	0.000186861111111111	,	5	)
(	0.000190472777777778	,	6	)
(	0.000240207777777778	,	7	)
(	0.0003920425	,	8	)
(	0.000412724166666667	,	9	)
(	0.00129943694444444	,	10	)
(	0.00149143027777778	,	11	)
(	0.00169629722222222	,	12	)
(	0.00219618333333333	,	13	)
(	0.00226506	,	14	)
(	0.00269566944444444	,	15	)
(	0.00371396388888889	,	16	)
(	0.00738218777777778	,	17	)
(	0.0119443702777778	,	18	)
(	0.012554175	,	19	)
(	0.0185923725	,	20	)
(	0.0244565030555556	,	21	)
(	0.0253227361111111	,	22	)
(	0.029150005	,	23	)
(	0.0331921175	,	24	)
(	0.0335283216666667	,	25	)
(	0.0403903283333333	,	26	)
(	0.0427223969444444	,	27	)
(	0.0503870433333333	,	28	)
(	0.0591365475	,	29	)
(	0.0659268655555556	,	30	)
(	0.0706740019444445	,	31	)
(	0.0736331347222222	,	32	)
(	0.19920361	,	33	)
(	0.222873552222222	,	34	)
(	0.244138505833333	,	35	)
(	0.340729438055556	,	36	)
(	0.346206156388889	,	37	)
(	0.351673278888889	,	38	)
(	0.376424763888889	,	39	)
(	0.512949354444444	,	40	)
(	0.637940470277778	,	41	)
(	0.668584187777778	,	42	)
(	0.871211547777778	,	43	)
(	1.05215060777778	,	44	)
(	1.41712999138889	,	45	)
(	1.51656711166667	,	46	)
(	1.6084562175	,	47	)
(	1.69678575305556	,	48	)
(	1.76299167222222	,	49	)
(	2.0762291125	,	50	)
(	2.26069241	,	51	)
(	2.40903645833333	,	52	)
(	5.38258680555556	,	53	)
(	6.43780164944444	,	54	)
(	6.98830132388889	,	55	)
(	9.49440755222222	,	56	)
(	13.2922504338889	,	57	)
(	14.1774490016667	,	58	)
(	16.9809537761111	,	59	)
(	17	,	60	)%
(	19	,	61	)%
(	24	,	61	)

	};
	\addlegendentry{\textbf{MTZ Clique Relaxation}}	

	\addplot[red,densely dotted] plot coordinates {		
	
(	0.0000275277777777778	,	1	)
(	0.000132214166666667	,	2	)
(	0.00016676	,	3	)
(	0.000177286666666667	,	4	)
(	0.000182900833333333	,	5	)
(	0.000259201388888889	,	6	)
(	0.000271471111111111	,	7	)
(	0.000289458055555556	,	8	)
(	0.000401732777777778	,	9	)
(	0.001566425	,	10	)
(	0.00346299805555556	,	11	)
(	0.00358237916666667	,	12	)
(	0.00618690638888889	,	13	)
(	0.00635992361111111	,	14	)
(	0.00648648694444444	,	15	)
(	0.0133207341666667	,	16	)
(	0.0134439669444444	,	17	)
(	0.01811928	,	18	)
(	0.0246942477777778	,	19	)
(	0.0295369233333333	,	20	)
(	0.0411796486111111	,	21	)
(	0.0927273644444444	,	22	)
(	0.123879055555556	,	23	)
(	0.125999883055556	,	24	)
(	0.128180313055556	,	25	)
(	0.155878838333333	,	26	)
(	0.160185377222222	,	27	)
(	0.215635867777778	,	28	)
(	0.221035393611111	,	29	)
(	0.225935295833333	,	30	)
(	0.226364525555556	,	31	)
(	0.239195285277778	,	32	)
(	0.374471672777778	,	33	)
(	0.437956983888889	,	34	)
(	0.639380289722222	,	35	)
(	0.722954983055556	,	36	)
(	0.803080308611111	,	37	)
(	0.943416544722222	,	38	)
(	3.19998345277778	,	39	)
(	3.89046468111111	,	40	)
(	5.38353515611111	,	41	)
(	5.80693250861111	,	42	)
(	6.13583550333333	,	43	)
(	6.94267957888889	,	44	)
(	7.30682725694444	,	45	)
(	7.72102593305556	,	46	)
(	7.86409125444444	,	47	)
(	8.86302951388889	,	48	)
(	14.2457486977778	,	49	)
(	14.5543847655556	,	50	)
(	17.4059613716667	,	51	)
(	18.8298502605556	,	52	)
(	24	,	52	)

	};
	\addlegendentry{\textbf{Connectivity Relaxation}}	
	
	\end{axis}
	\end{tikzpicture}
	\caption{Performance profile graph of \#solved instances using different pricing problem formulations for $n=$20-60 (70 instances).}  \label{performanceProfiles-all}
\end{figure}

To better understand the \BP algorithm performance, the reader can see in Table \ref{answer:table2} different parameters computed as averages on ten instances: the gap at the root node (RootNodeGap(\%)); the number of necessary variables (Total) split by the variables added  at the beginning (Initial), obtained trough the heuristic pricer (Heur), and given by the exact pricer (Exact); the number of times that this latter routine is called (ExactIter); the nodes of the master problem branch-and-bound tree (Nodes); and the percentage of CPU time that the algorithm uses to solve the pricing problem (PricingTime(\%)).

The results are shown until $n=50$ to focus on those instances that were solved to optimality. We are presenting  the results only for one of the five different formulations of the pricing problem, namely \textit{Flow Clique Relaxation formulation}, since all the others exhibit similar results. The reader should note that the parameters which are analyzed refer to the MP rather that to the pricing problem and therefore the formulation used in the pricing problem is not very important to explain their behaviour.

\begin{table}[htb]
\scriptsize
\centering
\begin{tabular}{|r|r|rrrr|r|r|r|}
\hline
 $n$&RootNodeGap(\%)&\bf Total& Initial & Heur & Exact & ExactIter  & Nodes& PricingTime(\%)\\
 \hline
$n=20$&	0.50	&\bf	80.70	&	53.40	&	23.80	&	3.50	&	3.50	&	1.60	&	97.95	\\
$n=30$&	0.13	&\bf	137.20	&	81.60	&	44.30	&	11.30	&	6.10	&	1.20	&	99.91	\\
$n=36$&	0.21	&\bf	234.40	&	100.90	&	103.70	&	29.80	&	12.00	&	1.60	&	99.98	\\
$n=40$&	0.00	&\bf	286.20	&	108.50	&	126.10	&	51.60	&	18.00	&	2.10	&	99.99	\\
$n=50$&	0.01	&\bf	443.40	&	136.50	&	226.60	&	80.30	&	34.80	&	1.30	&	100.00	\\

 \hline


\end{tabular}
\caption{Branch-and-price performance}
\label{answer:table2}
\end{table}

Here we can see the strength of the \BP algorithm: it can solve the problem using very few variables comparing with other already proposed formulations to solve the GCCP, which are of the order of $O(n^2)$ (see \cite{BPR17}). The goodness of the root node gap makes the size of the B\&B tree small. The counterpart is the time that the algorithm spends solving the pricing problem. To deal with it one has to save calls to the exact routine what is done by means of initial columns, the heuristic pricer, and adding several variables with negative reduced cost in each iteration.

\subsection{Comparing different heuristic algorithms on \GCPP}

The previous computational section shows that heuristic procedures are still needed to solve the largest instances of the \GCPP. Here, results about Algorithm \ref{algo:rs} (RS) are reported. Parameter $\theta = max\_random\_steps$ has been fixed such that $\theta \in \{\lceil|V|/3\rceil,\lceil |V|/2 \rceil, \lceil 2|V|/3\rceil\}$,
(the greater the number, the more the algorithm is driven by random choices of $e_{ij}$-s). The number of starting solutions is $max\_start = 10|V|$, the same number used in \cite{BPR17} to test heuristics Variable Neighborhood Search (VNS) and Random Restart (RR). That is, all algorithms try to improve the same number of initial solutions. For the test, algorithms were coded in Julia version 1.03 (\cite{bezanson2017julia}) and run on HP EliteBook with a Intel I5-core CPU.

Computational results are contained in Tables \ref{tab:medprob}, \ref{tab:modbigprob}, \ref{tab:bigprob} and \ref{tab:times}, in which  are reported the objective function and the first iteration ($it_b$) in which the best solution has been found (the largest the value, the most important is the diversification phase). For every $n$, we summarize the percentage gap to the optimal/best value and the average number of iterations.

Table \ref{tab:medprob} considers medium-sized problems for which we know the optimal objective function. Here we compare the new results of RS with the old ones obtained by RR, of VNS, and the optimal values, \cite{BPR17}. Note that some of those optimal solutions could not be found in \cite{BPR17} but are certified with the results of our \BP. It can be seen that average objective values are in favor of the new heuristic RS, as the average results of all three implementations are always better than the corresponding ones of both RR and VNS. If instead it is compared how many times the best solution is found, then it happened 26, 23, 22 for the three version of RS and only 20 times for RR and 9 times for VNS, so the same conclusion holds.

\begin{table}[!htb]
\begin{center}
\tiny
\begin{tabular}{|c|c|c|c|c|c|c|c|c|c|} \hline
Problem & fo[$|V|/3] $& $it_b$ & fo[$|V|/2$] & $it_b$ & fo[$2|V|/3$]  & $it_b$ & RR& VNS & Optimal Solution \\ \hline
 G20\_1&    -114 &     1 &    -114 &     1 &   -114 &     1 &   -114 & -114&   -114 \\
 G20\_2&     -74 &     2 &     -74 &    53 &    -74 &    46 &    -74 & -34&    -74 \\
 G20\_3&    -122 &    12 &    -122 &     7 &   -122 &    10 &   -122 & -122&   -122 \\
 G20\_4&    -112 &     2 &    -112 &     4 &   -112 &    16 &   -112 & -110&    -112 \\
 G20\_5&    -128 &     1 &    -128 &     1 &   -128 &     1 &   -128 & -102&    -128 \\
 G20\_6&    -102 &     2 &    -102 &    93 &   -102 &   125 &   -102 & -96&    -102\\
 G20\_7&    -154 &    11 &    -154 &    23 &   -154 &    23 &   -154 & -102&    -154 \\
 G20\_8&     -96 &   149 &     -94 &   188 &    -96 &    14 &    -94 & -92&     -96 \\
 G20\_9&    -116 &     1 &    -116 &     1 &   -116 &     1 &   -116 & -116&   -116 \\
 G20\_10&   -140 &     1 &    -140 &     1 &   -140 &     1 &   -140 & -140&   -140 \\
\hline																
$\bf n=20$&\bf	0.0	\%&\bf	118.2	&\bf	0.2	\%&\bf	138.6	&\bf	0.0	\%&\bf	124.0&\bf	0.2	\%&\bf	12.0	\%& \\
\hline																

 G30\_1&    -244 &   217 &    -248 &   293 &   -248 &   187 &   -254 & -248&    -254 \\
 G30\_2&    -152 &    33 &    -152 &    83 &   -152 &    22 &   -152 & -152&    -152 \\
 G30\_3&    -210 &   294 &    -206 &   197 &   -210 &   108 &   -200 & -144&    -210 \\
 G30\_4&    -200 &   191 &    -200 &     3 &   -192 &    57 &   -200 & -170&    -200 \\
 G30\_5&    -288 &   104 &    -288 &    87 &   -276 &    81 &   -288 & -276&    -288 \\
 G30\_6&    -260 &    29 &    -260 &    22 &   -260 &    13 &   -260 & -260&    -260 \\
 G30\_7&    -228 &   121 &    -228 &    25 &   -228 &   124 &   -228 & -222&    -228 \\
 G30\_8&    -126 &    29 &    -126 &    29 &   -126 &    22 &   -122 & -108&    -126 \\
 G30\_9&    -276 &   148 &    -274 &    66 &   -276 &   124 &   -276 & -136&    -276 \\
 G30\_10&   -174 &    57 &    -158 &     8 &   -174 &    14 &   -168 & -154&    -176 \\
\hline																
$\bf n=30$&\bf	0.5	\%&\bf	155.2	&\bf	1.5	\%&\bf	136.6	&\bf	1.2	\%&\bf	132.6	&\bf	1.2	\%&\bf	13.3	\%& \\
\hline																
																
 G36\_1&    -296 &    61 &    -300 &   135 &   -296 &   134 &   -296 & -296&    -300 \\
 G36\_2&    -304 &   140 &    -304 &   254 &   -304 &   188 &   -300 & -300&    -304 \\
 G36\_3&    -390 &    36 &    -390 &     3 &   -390 &   329 &   -356 & -340&    -390 \\
 G36\_4&    -336 &    13 &    -336 &     3 &   -336 &    83 &   -326 & -304&    -340 \\
 G36\_5&    -300 &    57 &    -300 &    42 &   -300 &   237 &   -300 & -300&    -300 \\
 G36\_6&    -286 &     1 &    -286 &     1 &   -286 &     1 &   -286 & -286&    -286 \\
 G36\_7&    -344 &   324 &    -320 &   143 &   -318 &   131 &   -310 & -324&    -344 \\
 G36\_8&    -230 &   143 &    -240 &    78 &   -230 &   117 &   -230 & -204&    -246 \\
 G36\_9&    -268 &    85 &    -246 &    34 &   -242 &     4 &   -260 & -242&    -268 \\
 G36\_10&   -290 &    41 &    -290 &   121 &   -290 &   350 &   -290 & -290&    -296 \\
 \hline																
$\bf n=36$&\bf	1.1	\%&\bf	129.7	&\bf	2.1	\%&\bf	127.1	&\bf	2.8	\%&\bf	150.0	&\bf	3.7	\%&\bf	6.1	\%& \\
\hline																

 G40\_1&    -306 &   300 &    -290 &   164 &   -292 &   134 &   -294 & -284&       -318 \\
 G40\_2&    -514 &   107 &    -514 &    54 &   -514 &   196 &   -514 & -508&    -514 \\
 G40\_3&    -306 &   174 &    -300 &    70 &   -306 &   184 &   -288 & -238&       -332 \\
 G40\_4&    -406 &   333 &    -394 &   285 &   -384 &    46 &   -384 & -384&    -412 \\
 G40\_5&    -342 &   196 &    -342 &   372 &   -342 &   141 &   -326 & -342&    -342 \\
 G40\_6&    -336 &   204 &    -326 &    19 &   -296 &    99 &   -292 & -252&       -336 \\
 G40\_7&    -306 &   377 &    -280 &    42 &   -306 &   344 &   -272 & -184&    -330 \\
 G40\_8&    -294 &   175 &    -314 &   111 &   -286 &     5 &   -270 & -252&    -314 \\
 G40\_9&    -374 &   170 &    -386 &   329 &   -354 &     1 &   -396 & -376&    -396 \\
 G40\_10&   -444 &   304 &    -444 &   255 &   -444 &   201 &   -420 & -372&    -456 \\
 \hline																
$\bf n=40$&\bf 3.5	\%&\bf	164.6	&\bf	4.6	\%&\bf	145.4	&\bf	6.4	\%&\bf	136.8	&\bf	8.5	\%&\bf	15.9	\%& \\
\hline

 mean                                                & -249.70 & 116.15& -247.45 & 92.50 &-245.40 & 97.88 &-242.85 &  -224.40 &-253.80   \\ \hline
\mbox{} \#solved & 26        &             & 23         &           & 22        &            & 20       & 9 &\\ \hline
\end{tabular}
\end{center}
\caption{Results on medium-sized problems for different heuristics.}
\label{tab:medprob}
\end{table}

The results on new instances, for which  in most cases the optimal solutions have been found in this research, are reported in Table \ref{tab:modbigprob}. It can be seen again that RS with any parameter is on average better than both VNS and RR, even though this time the heuristics have seldom found the optimal solutions. Best solutions have been found 15, 8, 11 times by the three RS's, 6 by RR, and 5 times by VNS. This shows that there is still room for improving the heuristic algorithms (see next section).

\begin{table}[htb]
\begin{center}
\tiny
\begin{tabular}{|c|c|c|c|c|c|c|c|c|c|} \hline
Problem & fo[$|V|/3]$& $it_b$& fo[$|V|/2$] & $it_b$ & fo[$2|V|/3$]  & $it_b$& RR& VNS & Optimal Solution \\ \hline
 G50\_1&  -494 &   399 &  -510 &   316 &  -510 &   390 &  -492 &  -472 &  -562 \\
 G50\_2&  -636 &   273 &  -636 &   225 &  -642 &    31 &  -642 &  -642 &  -650 \\
 G50\_3&  -610 &   326 &  -630 &    49 &  -664 &    74 &  -576 &  -576 &  -674 \\
 G50\_4&   -488 &    232 &   -470 &    173 &   -464 &     15 &   -446 &   -450 &   -504 \\
 G50\_5&   -644 &     33 &   -644 &   131 &  -644 &   213 &  -644 &  -644 &  -644 \\
 G50\_6&  -358 &    99 &  -380 &   202 &  -400 &    35 &  -298 &  -264 &  -400 \\
 G50\_7&  -536 &   435 &  -532 &   186 &  -534 &    14 &  -498 &  -498 &  -564 \\
 G50\_8&  -614 &    10 &  -614 &   418 &  -614 &    32 &  -622 &  -502 &  -674 \\
 G50\_9&  -638 &   322 &  -636 & 3 &  -636 &   283 &  -614 &  -586 &  -642 \\
 G50\_10&  -462 &   288 &  -446 & 3 &  -446 &   232 &  -446 &  -448 &  -464 \\
 \hline																
{$\bf n=50$}&\bf{	5.2	\%}&\bf{	142.8	}&\bf{	4.9	\%}&\bf{	129.9	}&\bf{	3.9	\%}&\bf{	123.1	}&\bf{	9.3	\%}&\bf{	12.6	\%}& \\
\hline																
 G54\_1&  -790 &    80 &  -788 & 1 &  -788 & 1 &  -788 &  -694 &  -790 \\
 G54\_2&  -588 &   393 &  -580 &   524 &  -580 &   150 &  -596 &  -522 &  -662 \\
 G54\_3&  -542 &    67 &  -542 &   206 &  -542 &    26 &  -510 &  -496 &  -544$^1$ \\
 G54\_4&  -560 &   146 &  -538 &   192 &  -544 &   382 &  -568 &  -446 &  -576 \\
 G54\_5&  -654 &    73 &  -654 &   513 &  -638 &   431 &  -650 &  -578 &  -670 \\
 G54\_6&  -568 &    61 &  -568 &   107 &  -568 &    80 &  -560 &  -564 &  -594 \\
 G54\_7&  -628 &   196 &  -624 &   105 &  -628 &   165 &  -614 &  -614 &  -640 \\
 G54\_8&  -606 &   181 &  -614 &   385 &  -606 &   104 &  -588 &  -578 &  -624 \\
 G54\_9&  -464 &   270 &  -450 &   261 &  -450 &    85 &  -394 &  -462 &  -490$^1$ \\
 G54\_10&  -804 &   368 &  -766 &   406 &  -800 &   378 &  -808 &  -808 &  -808 \\
 \hline																
{$\bf n=54$}&\bf{	3.2	\%}&\bf{	128.9	}&\bf{	4.4	\%}&\bf{	142.7	}&\bf{	4.2	\%}&\bf{	127.8	}&\bf{	5.6	\%}&\bf{	10.1	\%}& \\
\hline																
 G60\_1&  -726 & 5 &  -766 &   142 &  -682 &   199 &  -660 &  -604 &  -782$^1$ \\
 G60\_2&  -732 &   106 &  -732 &   168 &  -732 &   219 &  -636 &  -600 &  -732$^1$ \\
 G60\_3&  -836 &    85 &  -834 &    48 &  -828 &    83 &  -758 &  -832 &  -832 \\
 G60\_4&  -680 & 8 &  -684 &   596 &  -690 &   382 &  -666 &  -558 &  -750 \\
 G60\_5&  -666 &    39 &  -650 &   467 &  -660 &    36 &  -626 &  -666 &  -712 \\
 G60\_6&  -938 &    51 &  -938 &   469 &  -938 &   193 &  -836 &  -788 &  -964 \\
 G60\_7&  -562 &   257 &  -538 &   424 &  -518 &   534 &  -502 &  -500 &  -606 \\
 G60\_8&  -650 &   288 &  -620 &   484 &  -624 &   500 &  -582 &  -658 &  -664 \\
 G60\_9&  -894 &    17 &  -894 &   233 &  -858 &   305 &  -832 &  -848 &  -912 \\
 G60\_10&  -622 &   325 &  -604 &   255 &  -636 &   342 &  -602 &  -510 &  -682$^1$ \\ 
\hline																
{$\bf n=60$}&\bf{	4.5	\%}&\bf{	117.2	}&\bf{	5.3	\%}&\bf{	144.6	}&\bf{	6.4	\%}&\bf{	138.8	}&\bf{	12.4	\%}&\bf{	14.2	\%}& \\
\hline																
 mean & -633.00 &  181.10 & -629.40 & 256.40 & -628.80 & 197.13 & -601.80 & -580.27 &   -660.40 \\ \hline
\# best solution & 15& & 8& & 11& & 6 & 5 &\\ \hline
\end{tabular}
\end{center}
\footnotesize{$^1$ Best solution found by Algorithm \ref{algo:CG}}
\caption{Results on moderately large-sized problems for different heuristics.}
\label{tab:modbigprob}
\end{table}

Finally, in
Table \ref{tab:bigprob},  larger instances are considered, and again averages of objective values are in favor of the new heuristic, as the means of all three implementations are better than those from RR and VNS, and counting how many times the best known solution is found, respectively 14, 8,  2, 1, 2, is still in favor of RS. Regarding what parameter choice of RS is best, it can be seen that it does not make a great difference when the instance size is small, but when it gets larger, it seems that $max\_random\_steps =\lceil |V|/3  \rceil$ gives better results. Finally, the iteration in which the best solution is found (columns $it_b$) exhibits a great variability: This fact suggests that the diversification mechanism devised to explore different solutions has been effective.
Heuristics try to improve the same number of initial solutions and therefore they found the same number of local optima. So why RR and VNS, apparently more sophisticated, are left behind by RS? The reason could be the diversification. Both RR and VNS could be too constrained by the initial solution and they stop too early in inferior local optima.

\begin{table}[htb]
\begin{center}
\tiny
\begin{tabular}{|c|c|c|c|c|c|c|c|c|c|} \hline
Problem & fo[$|V|/3] $& $it_b$& fo[$|V|/2$] & $it_b$ & fo[$2|V|/3$]  & $it_b$& RR& VNS & Best Solution \\ \hline
 G80\_1&-1176 &780 &-1172 &171 &  -1164 &671 &  -1030 &-1086 & -1180 \\
 G80\_2&-1120 & 31 &-1096 &385 &  -1060 &667 &   -968 &-1042 & -1120 \\
 G80\_3&-1314 &220 &-1314 &361 &  -1290 &769 &  -1274 &-1260& -1314 \\
 G80\_4&-1066 &696 &-1036 &638 &  -1048 &780 &   -976 & -900  & -1078\\
 G80\_5&-1346 &213 &-1316 & 30 &  -1346 &568 &  -1234 &-1370  & -1370\\
 G80\_6& -956 & 46 & -956 &194 &   -930 &619 &   -936 & -818 & -1008 \\
 G80\_7&-1298 &193 &-1282 &163 &  -1270 &115 &  -1246 &-1286 & -1298 \\
 G80\_8&-1142 &636 &-1132 & 15 &  -1128 & 33 &   -998 & -904 & -1166\\
 G80\_9&-1368 &207 &-1368 &490 &  -1364 &  8 &  -1190 &-1196 & -1368 \\
 G80\_10&-1504 &730 &-1472 &489 &  -1472 &142 &  -1416 &-1440 & -1504 \\
 \hline																
{$\bf n=80$}&\bf{	1.0	\%}&\bf{	130.3	}&\bf{	2.2	\%}&\bf{	123.9	}&\bf{	2.9	\%}&\bf{	137.2	}&\bf{	9.3	\%}&\bf{	9.5	\%}& \\
\hline																
 G100\_1&-1732 &156 &-1732 &201 &  -1732 &446 &  -1630 &-1482 & -1746 \\
 G100\_2&-2126 &684 &-2110 &891 &  -2090 &868 &  -1730 &-1908 & -2126 \\
 G100\_3&-1544 &687 &-1544 &843 &  -1492 &346 &  -1216 &-1266 & -1544 \\
 G100\_4&-2184 &491 &-2208 &721 &  -2140 &166 &  -2094 &-1966 & -2208 \\
 G100\_5&-1708 & 19 &-1724 &198 &  -1690 &503 &  -1442 &-1386 & -1724 \\
 G100\_6&-2160 &678 &-2160 &295 &  -2160 &798 &  -2176 &-2176 & -2176 \\
 G100\_7&-1860 & 71 &-1838 & 44 &  -1914 &333 &  -1686 &-1756 & -1968 \\
 G100\_8&-1532 & 14 &-1506 &365 &  -1482 &937 &  -1390 &-1484 & -1532 \\
 G100\_9&-2090 &876 &-2084 &890 &  -2068 &309 &  -1934 &-1798 & -2090 \\
 G100\_10&-2276 &907 &-2308 &444 &  -2234 &586 &  -2136 &-2194 & -2308 \\
 \hline																
{$\bf n=100$}&\bf{	1.0	\%}&\bf{	122.6	}&\bf{	1.1	\%}&\bf{	125.4	}&\bf{	2.2	\%}&\bf{	128.0	}&\bf{	10.7	\%}&\bf{	10.7	\%}& \\
\hline																
 mean &-1575.10 &  416.75& -1567.90 &   391.40 & -1553.70 &   483.20 & -1435.10 & -1435.90 & -1591.40\\ \hline
\# best solution & 14 & & 8 & & 2& & 1 & 2& \\ \hline
\end{tabular}
\end{center}
\caption{Results on  large-sized problems for different heuristics.}
\label{tab:bigprob}
\end{table}

Regarding computational times, finding an improved solution with RS is much faster than with RR or VNS, and
the whole computational times are reported in Table \ref{tab:times}. As expected given the simplicity of the algorithm, the RS times are much less than VNS and RR. The reason is that the loop of Steps 9-17 of Algorithm \ref{algo:rs} is operated in $O(n^2)$, and it is repeated at most $O(n)$ times. Therefore it takes $O(n^3)$ operations to calculate a local optimum (to be repeated $\max\_\mbox{start}$ times).
Conversely, VNS and RR are based on local interchange, whose complexity is much higher. For example, it implies a subroutine of $O(n^2)$ only to check the connectivity of interchanging two units.

\begin{table}[htb]
\begin{center}
\small
\begin{tabular}{|c|r|r|r|} \hline
 $n$ & RS & RR & VNS \\ \hline
40 & 0.07 & 2.48 & 1.63\\
60 & 0.27 & 9.91 & 6.71 \\
80 & 0.75 & 30.15 & 19.32 \\
100 & 1.65 & 61.36 & 42.67\\ \hline
\end{tabular}
\end{center}
\caption{Average of computational times (in seconds).}
\label{tab:times}
\end{table}

Having found that the greedy descent performs much better than the local interchange, one may wonder to what extent this result may be applied to other constrained clique partitions. The result strongly depends on the computational cost of shrinking a node with respect to the cost of reassigning a unit. If connectivity constraints are replaced by community constraints, such as the one defined by modularity and/or cohesion, \cite{Cafieri2015}, then these are cases in which a reassignment affects the global properties of clusters. Conversely, the operations of shrinking nodes remains faster, therefore it is very likely that the solution space is explored more efficiently. 

\subsection{MILP-relaxed matheuristic for \GCPP: combining randomized shrink heuristic with branch-and-price}
\label{subsec:63}
The preceding methodologies can be combined for a
new matheuristic: MILP-relaxed matheuristic (Truncated Column Generation). As described previously, this matheuristic consists in solving the pricing problem heuristically with RS, while branching is permitted to the master problem. From previous computational tests, the method combines the velocity of RS with the global accuracy of an ILP formulation. The method is especially useful when an instance must be solved with sufficient accuracy.

In the following computational tests, we compare the solution quality of the RS heuristic with the MILP-relaxed matheuristic. Therefore, we report results on three heuristic: the plain RS, the MILP-relaxed matheuristic in which RS is called only to solve the pricing problem (Matheuristic), and the situation in which RS is also called to initialize the master with a feasible solution and to solve the pricing (RS $+$ Matheur.). Table \ref{t:gapheur} reports the results of these three algorithms on instances of sizes from $n=20$ until $n=1000$.
This table shows the CPU time in seconds (CPU) and  the gap (\textit{GAP=100(`this-heuristic-solution' - `best-solution')/$\,|\,$`best-solution'$\,|\,$}) of these three algorithms with respect to the best known solution for each instance (note that for instances of sizes greater than 60 the comparison is with respect to the best solution found by one of our own heuristics). The reader may observe that we have imposed to each run a maximum execution time of one hour. Times reported for RS $+$ Matheur. are the aggregation of the time running RS plus the time running matheuristic.
Our intuition is confirmed by data: On average the best results are obtained by the combination of RS with  MILP-relaxed matheuristic for any instance size.

\begin{table}[htb]
\centering
\small
\begin{tabular}{|r|r|r|r|r|r|r|}
\hline
	& \multicolumn{2}{c|}{RS}	& \multicolumn{2}{c|}{Matheuristic}	& \multicolumn{2}{c|}{RS $+$ Matheur.} \\
\hline
n & GAP & CPU & GAP & CPU & GAP & CPU \\
\hline
20&	0.21 &0.00 &	0.39 &	0.03 & 0.00 & 0.02\\
30&	1.25 &0.03 &	2.05 & 0.08 &	0.48 & 0.10 \\
36&	3.28 & 0.05 & 2.19 & 0.23 &	2.33 & 0.18 \\
40	&  	8.02 & 0.07 &	3.18 & 0.27 &	1.68 & 0.32 \\
50	& 	7.17 & 0.15  &	5.50 & 1.05 & 3.18 & 0.83 \\
54	&  2.98 & 0.18 &	5.75 & 0.82 & 2.91 & 0.62 \\
60	&  4.38 & 0.27 &	7.15 & 1.71 & 3.85 & 1.40 \\
80& 0.30 & 0.75 &	4.80 & 7.75 & 0.00 & 4.81\\
100 & 0.16 & 1.65 & 8.16 & 27.66 & 0.00 & 11.50 \\
200 & 0.22 &26.80 & 6.63 & 353.24 & 0.02 & 227.35 \\
500 & 0.06 & 1302.10& 4.40 & 3600.00 & 0.05 & 4902.10\\
1000 & 0.38 & 3600.00 & 3.85 & 3600.00 & 0.22 & 7200.00 \\
\hline
\bf Total Result  &\bf	2.37 &\bf 411.0 &\bf	4.50 & \bf 632.73	& \bf 1.23 &\bf	1029.10
 \\
\hline
\end{tabular}
\caption{Cpu time and \% gap of different heuristics with respect to best known solution. \label{t:gapheur}}
\end{table}


\section{Conclusions and future work}\label{sec:7}

This paper analyzes the Graph-Connected Clique-Partitioning Problem (\GCPP) presenting three different new solution approaches: one exact and two heuristics. In \cite{BPR17} this problem was already introduced but its solution methods could only handle small-sized instances.
Our new approaches improve this drawback. We provide a new Integer Linear Programming (ILP) formulation, based on a set partitioning formulation, that approximates very-well the unknown optimal solution. This set partitioning formulation is  solved implementing a branch-and-price (B\&P) algorithm. The resulting pricing problem is  a new combinatorial problem: the Maximum-weighted Graph-Connected Single-Clique (\GCSC), that is analyzed and solved proposing different MILP formulations.

Besides enlarging the sizes of problems that can be solved exactly,  to tackle larger size problems we propose two new fast heuristics:  the ``random shrink'' (RS) and a MILP-relaxed matheuristic. These algorithms improve the previous VNS  and RR algorithms of \cite{BPR17} since they are both faster and more accurate. Extensive computational experiments show the usefulness of our new approaches giving rise to new opportunities to apply this classification methodology, that combines individual and relational data to new actual situations.

\section*{Acknowledgements}
This research has been partially supported by Spanish Ministry of Education and Science/FEDER grant number  MTM2016-74983-C02-(01-02),  projects FEDER-US-1256951, CEI-3-FQM331, P18-FR-1422,  FEDER-UCA18-106895  by Junta Andaluc\'{\i}a/FEDER/UCA, and Contrataci\'on de Personal Investigador Doctor. (Convocatoria 2019) 43 Contratos Capital Humano L\'inea 2. Paidi 2020, supported by the European Social Fund and Junta de Andaluc\'ia. The authors also acknowledge funding from project \textit{NetmeetData}: Ayudas Fundaci\'on BBVA a equipos de investigaci\'on cient\'ifica 2019.


\appendix
\section{Alternative formulations for the pricing problem}
\label{appendixA}
\subsection{Flow-based formulation with a auxiliary node\label{ss:flowcon}}

The rationale of this formulation is the same to the one described in Section \ref{ss:flowsin}, but
using an auxiliary node as a source node. Let $G_D = (V \cup \{0\},A)$ be a digraph, in which there is an auxiliary node $\{0\}$ and a set of arcs,  $A$, so defined: Two arcs  $(i,j)$ and $(j,i)$ for every edge $e_{ij} (=e_{ji})\in E$; and the auxiliary arcs $(0,i)$ for all $i \in V$. Flow variables $f_{ij}$ are defined for all pairs $i,j$ such that $(i,j) \in A$, the node 0 is assumed to be the flow source node, and it is also assumed a demand of one flow unit from all nodes of $V$. To define this formulation, we use  the same set of variables used for $\mathbf{(F_{flow})}$ but taking into account that now the arc set $A$ includes the arcs with origin at 0.
%
This alternative flow-based formulation of \GCSC  is:
\begin{eqnarray}
\mathbf{(F^0_{flow})} & \min&\displaystyle\sum_{i \in V} \sum_{j \in V: j>i}^nc_{ij}z_{ij}
-\displaystyle\sum_{i \in V} \gamma_i^* x_i\\
\nonumber
& s.t. & \eqref{f1:1}-\eqref{f1:3}, \eqref{z_dom}-\eqref{x_dom}, \\
\label{f1:4}
& &f_{ij}+f_{ji}\le (n-1) z_{ij}, \quad \forall (i,j)\in A: \, i,j \in V,\, i<j, \\
\label{f1:6}
& &  \sum_{i\in V} f_{0i}=\displaystyle \sum_{i \in V} x_i,\\
\label{f1:7}
& &  f_{0i}+\sum_{j\in V  : \, (j,i)\in A} f_{ji}-
                  \sum_{j\in V :  \, (i,j)\in A} f_{ij} =x_i, \quad  \forall i \in V,\\
\label{f1:8}
& &z_{0i}\le x_i, \quad \forall i\in V,\\
\label{f1:9}
& &f_{0i}\le n z_{0i}, \quad  \forall i\in V,\\
\label{f1:10}
& & \sum_{i\in V}z_{0i}\le1,\\
\label{f1:11}
& &z_{0i}\in \{0,1\}, \quad \forall i\in V. 
\end{eqnarray}

Constraints \eqref{f1:4} avoid the flow between nodes which are not included in the optimal cluster $S$. Constraints \eqref{f1:6}-\eqref{f1:7} are the node conservation flow, in which one unit of flow is retained by the crossed node. Constraints \eqref{f1:8}-\eqref{f1:10} ensure that the outgoing flow from the auxiliary node is sent to at most one node of the cluster (the flow upper bound is $n$). Lastly,
\eqref{f1:11}  define the domain of the variables.


\subsection{Arborescence formulation with an auxiliary node \label{ss:mtzcon}}

 Let $G_D = (V \cup \{0\},A)$ be a digraph defined as in Subsection \ref{ss:flowcon}. The rationale  behind this formulation is the one followed in Subsection \ref{ss:mtzsin}, but the MTZ 
description of the Spanning Tree builds an arborescence rooted at an auxiliary node $0$. Binary variables $t$, $x$ and $z$ are defined as in the formulation $\mathbf{(F_{MTZ})} $ but taking into account that now the arc set $A$ includes the arcs with origin at 0. Hence, this alternative formulation
 of the minimum \GCSC  problem is:
\begin{eqnarray}
\nonumber
\mathbf{(F^0_{MTZ})} &\min& \sum_{i\in V} \sum_{j\in V:i<j} c_{ij}z_{ij}
-\displaystyle\sum_{i \in V} \gamma_i^* x_i \\
\nonumber
&s.t.&  \eqref{f1:1}-\eqref{f1:3}, \eqref{z_dom},\eqref{x_dom}- \eqref{f3:2},\eqref{f3:7},\eqref{f3:10}, \\
\label{f3:3}
& &  t_{0j}+\sum_{i\in V: (i,j) \in A } t_{ij}=x_j, \quad \forall j \in V,\\
\label{f3:4}
& & \sum_{j \in V} t_{0j}=1. 
\end{eqnarray}

Constraints \eqref{f3:3} and \eqref{f3:4} ensure there is only one incident arc to every node of the cluster, so variables $t_{ij}$ define a directed subtree. 

Formulation $\mathbf{(F^0_{MTZ})}$ can be strengthened with the some families of valid inequalities
 described in Subsection \ref{ss:valid-MTZ0} in Appendix.

\section{Valid inequalities for the pricing problem formulations}

\subsection{Valid inequalities for $\mathbf{(F_{flow})}$ \label{ss:valid-flow}}

Formulation $\mathbf{(F_{flow})}$ can be strengthened with the following family of valid inequalities:
\begin{eqnarray}
\label{f2:2}
& &  \sum_{j \in V: (i,j) \in A} f_{ij}  \ge \sum_{k \in V: k\le i} x_k-x_i-n \sum_{k \in V: k>i} z_{ik}-n(1-x_i), \quad \forall i \in V,  \\
\label{f2:4}
& & f_{ij} + f_{ji} \le (n-1)z_{ij}, \quad \forall (i,j) \in A: i<j,      \\
\label{f2:5}
& &f_{ij} + f_{ji}  \le (n-2)z_{ij}+\sum_{k\in V: i<k} z_{ik}, \quad \forall (i,j) \in A: i<j.
\end{eqnarray}

Constraints \eqref{f2:2} guarantee that the outflow from the node with the highest index of the cluster is at least the number of elements of the cluster minus one. Constraints \eqref{f2:4} and \eqref{f2:5} provide upper bound of the flow crossing an edge (in both sense) being this $n-1$ for the node with the
highest index in the cluster, $n-2$ for the remaining nodes in the cluster and 0 if one of the two end-nodes of the edges is not in the cluster.

\subsection{Valid inequalities for $\mathbf{(F_{MTZ})}$ \label{ss:valid-MTZ}}
Formulation $\mathbf{(F_{MTZ})}$ can be strengthened with the following set of valid inequalities:
\begin{eqnarray}
\label{f4:4}
& &    \sum_{i \in V: (i,j) \in A}   t_{ij}\le x_j, \quad \forall j \in V, \\
\label{f4:5}
& &   \sum_{i \in V, (i,j) \in A}          t_{ij}\le \sum_{k\in V: k>j}^n z_{jk}, \quad
 \forall j\in V, \\
 \label{f4:6}
& &  \ell_j \le(n-1) \sum_{k \in V: (k,j) \in A} t_{kj}, \quad \forall j \in V,\\
 \label{f4:7}
& &  \ell_i \ge \sum_{k \in V: (k,i) \in A} t_{ki}, \quad \forall i \in V.
\end{eqnarray}
Constraints \eqref{f4:4} guarantee that there is at most an incident arc in the nodes of the cluster. Constraints \eqref{f4:5} ensure that the node with the greatest index (the root of the subtree) does not have incoming arcs. Constraints \eqref{f4:6}-\eqref{f4:7} impose bounds on the $\ell$-variables.

\subsection{Valid inequalities for $\mathbf{(F^0_{MTZ})}$ \label{ss:valid-MTZ0}}

Formulation $\mathbf{(F^0_{MTZ})}$ can be strengthened with the following family of valid inequalities:
\begin{eqnarray}
\label{f3:5}
& &  t_{0j}+ z_{ij}\le x_j,  \quad \forall (i,j) \in A: i<j, \\
\label{f3:8}
& & \ell_i \ge x_i-t_{0i}, \quad \forall i \in V, \\
\label{f3:9}
&&  \ell_j \le(1-t_{0j})(n-1), \quad \forall j \in V.
\end{eqnarray}

Constraints \eqref{f3:5} establish that the
fictitious node is connected with the node of the greatest index of $S$, therefore they break up symmetric optimal solutions.
Constraints \eqref{f3:8}-\eqref{f3:9} establish valid bounds for the $\ell$-variables.

\end{document}